\documentclass[11pt]{preprint}
\usepackage[utf8]{inputenc}
\usepackage[full]{textcomp}
\usepackage[osf]{newtxtext} 

\usepackage{amsthm}
\usepackage{amssymb}
\usepackage{hyperref}
\usepackage{breakurl}
\usepackage{color}
\usepackage{orcidlink}

\usepackage{tikz}
\usepackage{pgfplots}

\usepackage{microtype}
\usepackage{mathtools}
\usepackage{comment}
\usepackage{mhequ}
\usepackage{mhsymb}

\usepackage{dsfont}

\usetikzlibrary{decorations.markings}
\usetikzlibrary{external}

\tikzset{->-/.style={decoration={
  markings,
  mark=at position #1 with {\arrow{>}}},
  postaction={decorate}}}

\makeatletter 
\newcommand*{\bigcdot}{}
\DeclareRobustCommand*{\bigcdot}{%
  \mathbin{\mathpalette\bigcdot@{}}%
}
\newcommand*{\bigcdot@scalefactor}{.5}
\newcommand*{\bigcdot@widthfactor}{1.15}
\newcommand*{\bigcdot@}[2]{%
  \sbox0{$#1\vcenter{}$}
  \sbox2{$#1\cdot\m@th$}%
  \hbox to \bigcdot@widthfactor\wd2{%
    \hfil
    \raise\ht0\hbox{%
      \scalebox{\bigcdot@scalefactor}{%
        \lower\ht0\hbox{$#1\bullet\m@th$}%
      }%
    }%
    \hfil
  }%
}
\makeatother

\colorlet{darkblue}{blue!90!black}
\colorlet{darkred}{red!90!black}
\colorlet{darkgreen}{green!70!black}

\def\f#1#2{\textstyle{#1\over #2}}

\def\xx {{\boldsymbol x}}
\def\eps{\varepsilon}
\def\e{{\rm e}} 
\def\dd{{\rm d}}

\def\sp{{\rm span}}

\def\de{{\partial}}
\def\id{\mathord{\mathrm{id}}}

\newtheorem{proposition}{Proposition}[section]
\newtheorem{theorem}[proposition]{Theorem}
\newtheorem{corollary}[proposition]{Corollary}
\newtheorem{lemma}[proposition]{Lemma}
\theoremstyle{definition}

\newtheorem{remark}[proposition]{Remark}

\numberwithin{equation}{section}

\title{A noise-induced transition in the Lorenz system}

\author{Michele Coti Zelati\orcidlink{0000-0002-2495-2212} and Martin Hairer\orcidlink{0000-0002-2141-6561}}
\institute{Department of Mathematics, Imperial College London, London, SW7 2AZ, UK\\
\email{m.coti-zelati@imperial.ac.uk, m.hairer@imperial.ac.uk}
}
\begin{document}



\maketitle

\begin{abstract}
We consider a stochastic perturbation of the classical Lorenz system in the range of parameters for which the
origin is the global attractor. 
We show that adding noise in the last component causes a transition from a unique 
 to exactly two ergodic invariant measures. The bifurcation threshold depends 
 on the strength of the noise: if the noise is weak, the only invariant measure is Gaussian,
 while strong enough noise  causes the appearance 
  of a second ergodic invariant measure.

\end{abstract}

\tableofcontents

\section{Introduction}

The classical Lorenz system \cite{LorenzOriginal} is a very popular prototypical toy model
for chaos / turbulence \cite{LorenzTurbulence}. The traditional way of writing this system is given by 
\begin{equ}[e:LorenzDet]
\dot{X}=\sigma(Y-X)\;,\qquad
\dot{Y}=X(\rho-Z)-Y\;,\qquad 
\dot{Z}=-\beta Z+XY\;,
\end{equ}
with parameter values $\sigma = 10$ and $\beta = 8/3$. Changing the value $\rho$ allows to 
explore a variety of different behaviours \cite{BookColin,OrderChaos}.
In particular, for $\rho < 1$, \eqref{e:LorenzDet} admits a unique fixed point at the origin which eventually 
attracts every single solution, while for $\rho > 1$ it admits two further `non-trivial' fixed points.
These fixed points become unstable at $\rho = \sigma{3+\beta+\sigma \over \sigma - \beta - 1} \approx 24.74$, 
after which the system exhibits either a chaotic attractor or stable limit cycles, see for example \cite{Exists}.
For very large values of $\rho$ ($\rho \gtrsim 313$), the system admits a stable limit cycle
which undergoes a cascade of period-doubling bifurcations as one decreases $\rho$. 

The vertical axis $H = \{(X,Y,Z)\,:\, X=Y=0\}$ is invariant for all parameter values and, 
for $\beta > 0$, it is contained in the stable manifold of the origin.
The aim of this article is to explore how \eqref{e:LorenzDet} behaves in the stable regime $\rho < 1$
under the addition of noise to the $Z$-component.
More precisely, we consider the modified system
\begin{equ}[e:Lorenz]
\dot{X}=\sigma(Y-X)\;,\qquad
\dot{Y}=X(\rho-Z)-Y\;,\qquad
\dot{Z}=-\beta Z+XY + \hat\alpha\,\xi\;,
\end{equ}
where $\hat\alpha > 0$ and $\xi$ denotes white noise. 
For all values of $\hat\alpha$, this system admits as invariant measure the measure $\nu_0$ with $\nu_0(H) = 1$ under which
$Z \sim \CN(0,\hat\alpha^2/2\beta)$. 
The question we consider is whether it admits other invariant measures supported in $\R^3 \setminus H$. 
Our main result then states that while $\nu_0$ is the unique invariant measure for
\eqref{e:Lorenz} for small values of $\hat \alpha$, it necessarily admits a second invariant measure for large values.

\begin{theorem}\label{theo:main}
For any $\sigma, \beta > 0$ and any $\rho < 1$, there exist values 
$0 < \alpha_\star \le \alpha^\star < \infty$ such that
\begin{enumerate}
\item For $0 \le \hat\alpha < \alpha_\star$, \eqref{e:Lorenz} admits $\nu_0$ as its unique invariant measure.
\item For $\hat\alpha > \alpha^\star$, \eqref{e:Lorenz} admits exactly two ergodic 
invariant measures: $\nu_0$ and another measure $\nu_\star$. Furthermore,
$\nu_\star$ has a smooth density with respect to Lebesgue measure on $\R^3 \setminus H$ and there exists $\kappa > 0$
such that $\int (x^2 + y^2)^{-\kappa} \nu_\star(\dd x,\dd y,\dd z) < \infty$.
\end{enumerate}
For $\rho \ge 1$, there exists $\alpha^\star \ge 0$ such that the second statement still holds.
\end{theorem}

\begin{proof}
The fact that for $\hat \alpha > 0$ \eqref{e:Lorenz} admits at most one ergodic invariant
measure besides $\nu_0$ is the content of Theorem~\ref{theo:basic}.
Theorem~\ref{theo:Lyapunov} links the existence of the additional invariant measure $\nu_\star$
to the sign of the quantity $\lambda_\alpha$ whose asymptotic behaviour for both small and large
values of $\hat \alpha$ is obtained in Theorem~\ref{theo:behaviourLyap}. It remains to note that 
one always has $\lambda_\alpha > 0$ (and therefore existence of $\nu_\star$) for $\hat \alpha$
large enough, while its sign as $\hat \alpha \to 0$ is negative when $\rho < 1$ and positive
when $\rho > 1$. 
\end{proof}

\begin{figure}
\begin{center}
\input{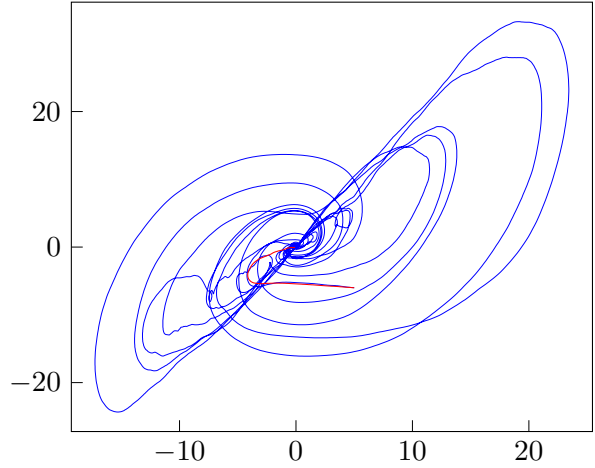}
\end{center}
\caption{The blue trajectory shows a simulation of~\eqref{e:Lorenz} in the $(x,y)$ plane with 
$\sigma = 10$, $\beta = 8/3$, $\rho = 1/2$, and $\hat\alpha = 30$. 
The red trajectory is obtained by using the same parameters, initial condition,
and realisation of the noise, except that $\hat\alpha = 10$.}\label{fig:simul}
\end{figure}

\begin{remark}
One would naturally expect to have $\alpha_\star = \alpha^\star$, but we cannot guarantee this at the moment.
It does however follow from our analysis that, for all values of $\hat\alpha$, \eqref{e:Lorenz} admits at most one
ergodic invariant measure besides $\nu_0$ and that the set of values $\hat\alpha$ for which $\nu_0$ is the unique
invariant measure consists of finitely many intervals. 
One would also expect to have $\alpha^\star = 0$ when $\rho> 1$ since then $H$ is already
linearly unstable for the deterministic system, but although our results do indeed guarantee that
in this case the system admits a unique additional invariant measure $\nu_\star$ for all 
$\hat \alpha \in (0,\delta)$ for some $\delta > 0$, we cannot rule out the existence of 
an intermediate range of values for which $\nu_0$ would be the unique invariant measure.
\end{remark}

\begin{remark}
One motivation for the study of \eqref{e:Lorenz} is that this provides a toy model
for the following situation. Consider the 2D Navier--Stokes equations on a torus with 
additive translation-invariant Gaussian forcing. It was shown in \cite{NSErgodic,HMGap}
(see also the earlier works \cite{EMS,BKL,KS} showing similar results under stronger non-degeneracy assumptions) 
that this system admits a unique invariant measure under a very weak
non-degeneracy assumption on the noise. This result however fails to apply to the situation where
the noise is itself periodic with a period strictly smaller (say half) than that of the torus on which the
system is posed. In this case, we know by \cite{NSErgodic} that it admits a unique invariant measure concentrated on functions
with the same period as the noise and, by \cite{JonathanDissipative}, that this measure is unique at small enough Reynolds number.
A natural question then is whether a bifurcation appears at high Reynolds number, 
as one would expect
from heuristic considerations. Our result can be viewed as a first mathematically rigorous result pointing 
in the direction of a positive answer. 
\end{remark}

\begin{remark}
The behaviour observed here can be contrasted with that observed in \cite{NoiseStable,NoiseStable2}, where the
authors exhibit a system with a quadratic nonlinearity which can explode in finite time in the absence
of noise but admits global solutions (and a unique invariant measure) in the presence of noise.
\end{remark}

Figure~\ref{fig:simul} shows a simulation of \eqref{e:Lorenz} for $\rho = 1/2$ and two different values of $\hat\alpha$.
For $\hat\alpha > \alpha^\star$, a typical trajectory consists of relatively long stretches of time spent in the 
vicinity of $H$, interspersed with excursions away from $H$. These excursions all escape in roughly the same direction
and the $Z$-coordinate (not depicted in the figure) is always quite negative when this happens. When $Z$ becomes positive,
they then spiral back in towards $H$. This can be understood by a linear analysis of the two-component system for 
a fixed value of $Z$: for $Z < \rho -1$ this system has one stable and one unstable direction, while it is stable
for all other values, exhibiting oscillations when $Z > \rho + (1-\sigma)^2 / (4\rho)$.  A good approximation of this behaviour
can be understood in terms of averaging of the eigenvalues of the first two equations of \eqref{e:Lorenz} with respect to the
Gaussian measure $\nu_0$, although this is not completely correct (see Remark \ref{rem:heuristics}).

\begin{remark}
The recent work \cite{Zdzislaw} analyses a toy model very similar to ours and with the
same underlying motivation. In our notations, their toy model reads
\begin{equ}[e:Simple]
\dot{X}=-X(1-Z)\;,\quad
\dot{Y}=-Y(\rho -Z)\;,\quad
\dot{Z}=-\beta Z - (X^2 + Y^2) + \kappa +\hat\alpha\,\xi\;.
\end{equ}
Although this appears on the face of it to be a three-dimensional model, 
it is effectively two-dimensional: if $\rho = 1$, then the ratio $X/Y$ remains constant in time. Otherwise, one has
\begin{equ}
Y(t) = {Y_0 \over X_0} X(t) \,\e^{(1-\rho)t}\;,
\end{equ}
so that one can reduce oneself to $Y = 0$ when $\rho > 1$ and $X=0$ when $\rho < 1$.
We deduce from this simple observation a slight strengthening of~\cite[Thm~3.6]{Zdzislaw}, namely that if $\kappa > \beta (1\wedge \rho)$,
then \eqref{e:Simple} admits \textit{exactly} $3$ ergodic invariant measures when $\rho \neq 1$
and uncountably many ergodic invariant measures when $\rho = 1$. Note that this
does not depend on $\hat\alpha$, it is in particular also true for $\hat\alpha = 0$ in which
case the invariant measures are concentrated on fixed points. 

A version of the Lorenz system in which the parameter $\rho$ randomly switches between two ``unstable'' values 
(28 and close to 28) was  analysed in \cite{BH12,Strickler}.  In these works, the existence of exactly two invariant 
measures was proved by combining an extension of H\"ormander condition
to piecewise deterministic systems and the precise knowledge of the Lyapunov exponents of the system.
\end{remark}

\begin{remark}
A general framework for verifying the stability / instability for an invariant subset of a Markov process was 
given in~\cite{Benaim} in terms of existence of suitable Lyapunov-type functions. From this point of view,
our approach is much less sophisticated, as we simply look for a function $V \colon \R^3 \setminus H \to \R_+$
with compact sublevel sets and such that $\CL V \le K - c V$. The main difficulty in our case is to be able
to actually build a Lyapunov function when $\alpha$ is large and to show that this construction must break down at some sufficiently
low value of $\alpha$. 
\end{remark}

The structure of the remainder of this article is as follows. First, in Section~\ref{sec:Notation} we perform
a simple change of variables that brings \eqref{e:Lorenz} in a slightly more canonical form
and we introduce some notation. In Section~\ref{sec:HypoControl}, we then provide a preliminary analysis 
of the equation which shows that it is irreducible and strong Feller on $\R^3 \setminus H$, so that
in particular it can have at most one additional ergodic invariant measure $\nu_\star$ besides $\nu_0$.
The core of our analysis is contained in the last two sections. First, in Section~\ref{sec:Lyapunov}, we construct
a Lyapunov function which allows to reduce the existence / non-existence of $\nu_\star$ to the 
behaviour of the invariant measure $\mu_\alpha$ for an auxiliary problem describing the behaviour
of a ``linearised'' version of \eqref{e:Lorenz} around $H$. The construction of the Lyapunov function
uses a philosophy similar to that used in \cite{Slow,HowHot}, namely to exhibit a ``fast'' dynamic in the regime
of interest and to use this to build a ``corrector'' which then allows to turn a ``na\"\i ve'' Lyapunov function
for the effective ``slow'' dynamic into a proper Lyapunov function for the full system.
Finally, Section~\ref{sec:largealpha}
analyses the behaviour of $\mu_\alpha$ as $\alpha \to \infty$, which allows us to conclude that 
\eqref{e:Lorenz} is indeed destabilised for any value of its parameters provided that $\alpha$ is 
sufficiently large. Since $\mu_\alpha$ describes a non-equilibrium system, it is not explicit and 
moreover has a quite complicated structure. The study of its asymptotic behaviour as $\alpha\to\infty$ 
is therefore highly non-trivial  and constitutes one of the main points of this article.

\subsection*{Acknowledgements}

{\small
MCZ gratefully acknowledges support by the Royal Society through a university research fellowship.
MH gratefully acknowledges support by the Royal Society through a research professorship.
}

\section{Notations}
\label{sec:Notation}

It will be convenient to write \eqref{e:Lorenz} in such a way that all of the 
arbitrary constants appear in the equation for $Z$. This will be convenient since 
we will be mostly interested in the regime where $x^2 + y^2 \ll 1$, so that 
$Z$ is close to a simple Ornstein--Uhlenbeck process. 
To this end, we define the constants
\begin{equ}
\chi = {2\over 1+\sigma}\;,\qquad \eta = {1+\sigma \over 2\sigma}\;,\qquad
\gamma = \chi \beta\;,\qquad \nu^2 = \chi^5\sigma\;,\qquad
\alpha =  \nu \sqrt\sigma \hat\alpha\;.
\end{equ}
as well as
\begin{equ}
\qquad z_\star = 2 + \chi^2 \sigma(\rho-1)\;.
\end{equ}
If we then perform the change of variables
\begin{equ}\label{e:newVar}
x(t) = {\nu \over \chi}  X(\chi t)\;,\qquad
y(t) = \nu \sigma \bigl(Y(\chi t) - X(\chi t)\bigr)\;,\qquad
z(t) =  z_\star - \chi^2 \sigma Z(\chi t)\;,
\end{equ}
the system \eqref{e:Lorenz} can be rewritten as
\begin{equ}[e:LorenzNew]
\dot{x}=y\;,\quad
\dot{y}=x(z-2)-2 y\;,\quad
\dot{z}=-\gamma (z - z_\star) + \alpha\,\xi - x(x+ \eta y)\;.
\end{equ}
We also introduce ``polar coordinates''
\begin{equ}
x =  \e^r \sin \theta\;,\qquad y = \e^r \left(\cos\theta - \sin\theta\right)\;,
\end{equ}
so that one can alternatively write the equations of motion as
\begin{equ}[e:equationR]
\dot\theta = 1- z \sin^2(\theta)\;,\qquad \dot r = -1 + \frac{z}{2} \sin (2\theta) \;.
\end{equ}
(These coordinates are chosen in such a way that \eqref{e:equationR} is as simple
as possible when $z = 0$ and can be ``guessed'' by looking at the explicit solution
of the damped harmonic oscillator describing the $(x,y)$ system with $z=0$.)
An important role will be played by the ``linearisation'' obtained by replacing 
the last equation in \eqref{e:LorenzNew} by the  Ornstein--Uhlenbeck process
\begin{equ}[e:defOU]
\dot{z}=-\gamma (z - z_\star) + \alpha\,\xi\;.
\end{equ}
We will use the notation $\CL$ for the generator of \eqref{e:LorenzNew} and $\CL_1$
for the generator of the ``linearised'' system, namely
\minilab{e:gen}
\begin{equs}
\CL &= y \d_x + \bigl(x(z-2)-2 y\bigr)\d_y - \bigl(x(x+ \eta y)+\gamma (z - z_\star)\bigr)\d_z
+ {\alpha^2\over 2}\d_z^2\;,\qquad \label{e:genL}\\
\CL_1 &= y \d_x + \bigl(x(z-2)-2 y\bigr)\d_y -\gamma (z - z_\star)\d_z
+ {\alpha^2\over 2}\d_z^2\;.\label{e:genL1}
\end{equs}
We will also use $\CL_0$ for the generator of the $(\theta,z)$-component of the linearised
system, namely
\minilab{e:gen}
\begin{equ}[e:genL0]
\CL_0 = \bigl(1- z \sin^2(\theta)\bigr)\d_\theta - \gamma (z - z_\star)\d_z
+ {\alpha^2\over 2}\d_z^2\;.
\end{equ}
We henceforth fix the values of the constants $\eta$, $\gamma$ and $z_\star$ appearing in 
our dynamic, but we will keep track on the dependence on $\alpha$.

In particular, we write $\mu_\alpha$ for the invariant measure 
on $S^1 \times \R$ for the diffusion with generator $\CL_0$. Such an invariant measure clearly
exists by Krylov--Bogoliubov. It is also quite easy to see that it is 
unique as a consequence of the controllability result shown in Proposition~\ref{prop:control} below
and the regularity result given by Proposition~\ref{prop:Hormander}. (See Theorem~\ref{theo:basic} for
a reference.)

\begin{remark}\label{rem:heuristics}
A ``na\"\i ve'' heuristic for the stability of $H$ goes as follows. Writing $\lambda_+(z)$ for the
largest real part of the eigenvalues of the system $(x,y)$ given in \eqref{e:LorenzNew} (with $z$ frozen),
one can verify that $\lambda_+(z) = -1 + \sqrt{z-1}\one_{z > 1}$.
This then suggests that $\alpha^\star$ is the smallest value such that
\begin{equ}
\E_\alpha \big(\sqrt{z-1}\one_{z > 1}\big) \ge 1\;,
\end{equ}
where the expectation is taken over the invariant measure $\CN(z_\star, \alpha^2/(2\gamma))$
for \eqref{e:defOU}. While this heuristic is incorrect, it is quite accurate in practice. For example,
for the parameters used in Figure~\ref{fig:simul}, it suggests $\alpha^\star \approx 27.04$
while numerical simulations suggest $\alpha^\star \approx 27.7$.
\end{remark}

\section{Hypoellipticity and control}
\label{sec:HypoControl}
The goal of this section is to analyse irreducibility and regularity properties on $\R^3 \setminus H$ of our stochastic Lorenz system.
We will prove the following result.
\begin{theorem}\label{theo:basic}
For every value of its parameters, \eqref{e:LorenzNew} admits at least one and at most two
ergodic invariant probability measures.
\end{theorem}

\begin{proof}[Proof of Theorem \ref{theo:basic}]
Since the restriction of our system to $H$ is simply an Ornstein--Uhlenbeck process,
the Gaussian with variance $\alpha^2 / (2\gamma)$ centred at $z_\star$ is the unique ergodic invariant measure there.
On $\R^3 \setminus H$, one has at most
one invariant probability measure since Proposition~\ref{prop:Hormander} implies that the transition
probabilities are strong Feller, while Proposition~\ref{prop:control}, combined
with the support theorem~\cite{SV} implies that every point of $\R^3 \setminus H$ is accessible. 
We conclude by for example \cite[Cor.~7.8]{ErgodicSPDE}.
\end{proof}

The fact that the transition
probabilities are strong Feller is contained in the following proposition.

\begin{proposition}\label{prop:Hormander}
The system \eqref{e:LorenzNew} satisfies H\"ormander's condition on $\R^3 \setminus H$.
\end{proposition}
%

\begin{proof}
We want to show that the $\CC^\infty$-module $\CM$ generated by
the iterated Lie brackets of the two vector fields
\begin{equ}
X_0= y\d_x + (x(z-2)-2 y)\d_y
-\big(\gamma (z - z_\star)  + x(x+ \eta y)\big)\d_z\;,
\quad
X_1= \alpha\d_z\;,
\end{equ}
so that $\CL = X_0 + {1\over 2}X_1^2$,
is of maximal rank at every point of $\R^3 \setminus H$.
A simple calculation shows that their Lie bracket is given by
\begin{equ}
X_2 =[X_0,X_1] =-  \alpha x \d_y +\alpha\gamma \d_z\;, 
\end{equ}
and that furthermore,
\begin{equ}
X_3 =[X_0,X_2]=
\alpha x	\d_x
-\alpha (y+(2+\gamma)x) \d_y
-\alpha \left(\eta x^2-\gamma^2\right)\d_z\;.
\end{equ}
Note that $\sp\{X_1,X_2,X_3\}=\R^3$ whenever $x\neq 0$,
while $\sp\{X_0,X_1,X_3\}=\R^3$ when $x=0$ and $y\neq 0$, and therefore 
H\"ormander's condition is satisfied everywhere on $\R^3 \setminus H$.
\end{proof}

The next result concerns controllability properties of the Lorenz system, and, in particular, accessibility to any point in $\R^3 \setminus H$.

\begin{proposition}\label{prop:control}
Given any initial condition $(x_0,y_0,z_0)\in \R^3 \setminus H$, any target point $(\bar{x},\bar{y},\bar{z})\in \R^3 \setminus H$, and any arbitrary $\eps>0$, there exists $T=T(|x|,|y|,|z|,\eps)\geq 0$ and a function $h\in \CC^1([0,T],\R)$ 
such that the unique solution to 
\begin{equ}[e:LorenzControl]
\dot{x}=y\;,\quad
\dot{y}=x(z-2)-2 y\;,\quad
\dot{z}=-\gamma (z - z_\star) - x(x+ \eta y)+h\;,
\end{equ}
with initial condition $(x_0,y_0,z_0)$ satisfies
$
|(x(T),y(T),z(T))-(\bar{x},\bar{y},\bar{z})|<\eps
$.
\end{proposition}

\begin{proof}
Since $z$ in \eqref{e:LorenzControl} can be completely controlled by $h$, the proof relies on finding
a smooth curve $z$ so that $\xx_0=(x_0,y_0)$ and $\bar{\xx}=(\bar{x},\bar{y})$ can be connected by a solution
of the first two equations in \eqref{e:LorenzControl}.
The proof is divided in three steps.

\medskip

\noindent\emph{Step 1.} We first identify a non-smooth trajectory that connects $\xx_0$ with $\bar{\xx}$, 
as depicted in Figure \ref{fig:control}. 
Notice that for $\zeta\in \R$ given, the first two equations of \eqref{e:LorenzControl} constitute a two-dimensional linear system, that
can be written as
\begin{equ}
\begin{pmatrix}
\dot{x}\\
\dot{y}
\end{pmatrix}
=A
\begin{pmatrix}
x\\
y
\end{pmatrix},
\qquad
A=
\begin{pmatrix}
0& 1\\
\zeta-2  & -2
\end{pmatrix}.
\end{equ}
The eigenvalues of $A$ are
\begin{align}\label{eq:eigenvaluesA1}
\lambda_{1}(A)= -1 - \sqrt{\zeta-1} ,
\qquad \lambda_{2}(A)=-1 + \sqrt{\zeta-1}.
\end{align}
The idea is now to alternate between the globally stable dynamics (when $\zeta<1$)  for which $(0,0)$
is attractive, and the case in which there are a stable and an unstable manifold (when $\zeta>2$). 
We therefore choose $\zeta=0$, with corresponding matrix $A_0$, to implement the first scenario, 
and $\zeta=5$, with corresponding matrix $A_1$, to implement the second scenario. Notice that
$A_1$ has eigenvalues
\begin{align}
\lambda_{1}(A_1)= -3,
\qquad \lambda_{2}(A_1)=1.
\end{align}
 with corresponding eigenvectors
\begin{align}
e_{1}= \left(-1,3\right),\qquad e_{2}=\left(1,1\right),
\end{align}
so that the diagonal $\ell_u=\{(x,y):x=y\}$ is the unstable manifold of the system.
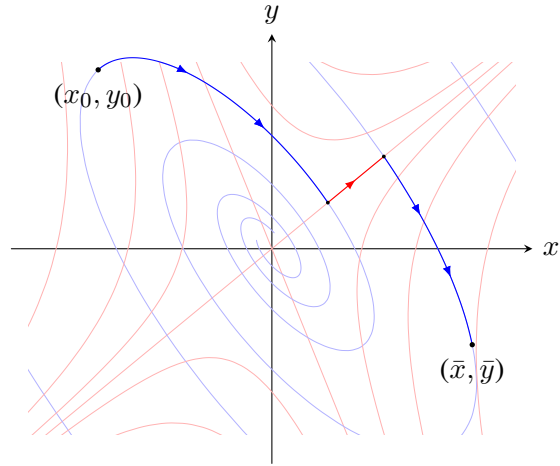
\begin{figure}[h!]
\begin{center}
\begin{tikzpicture}[cap=round,>=latex]
    \begin{axis}[ 
axis y line=center,
axis x line=middle,
axis line style = thin,
xlabel=$x$,
ylabel=$y$,
xmax=3,xmin=-3,
ymin=-3,ymax=3,
domain=-3:3,
samples=200,
no marks,
every axis x label/.style={
    at={(ticklabel* cs:1)},
    anchor=west,
},
every axis y label/.style={
    at={(ticklabel* cs:1)},
    anchor=south,
},
ticks=none
	]
	
\begin{scope}
	\clip (axis cs:-2.8,-2.6) rectangle (axis cs:2.8,2.6);

   	\addplot [domain=-3:3, samples=5, red!30, very thin] (x, x); 
	\addplot [domain=-2:2, samples=5, red!30, very thin] (x, -3*x); 

	\addplot [domain=-5:12, samples=100, smooth,very thin, blue!30] ({exp(-0.2*x)*(-2*cos(deg(x)) + 0.5*sin(deg(x)))}, {exp(-0.2*x)* (2.5*cos(deg(x)) +1.5*sin(deg(x)))}); 
	\addplot [domain=-4:14, samples=100,smooth, very thin,blue!30] ({2*exp(-0.2*x)*(-2*cos(deg(x)) + 0.5*sin(deg(x)))}, {2*exp(-0.2*x)* (2.5*cos(deg(x)) +1.5*sin(deg(x)))}); 

\foreach \i in {-1.2,-0.8,-0.5,0.5,0.8,1.2} 
{
\foreach \j in {-1,1}
{
	\addplot [domain=-3:15, samples=30,smooth, very thin,red!30] ({\i*(exp(0.2*x) + \j*exp(-0.3*x))}, {\i*(exp(0.2*x) - 3*\j*exp(-0.3*x))}); 
}}
\end{scope}
	\addplot [->-=0.33,->-=0.66,domain=1.7894:2.5, samples=200, blue] ({2*exp(-0.2*x)*(-2*cos(deg(x)) + 0.5*sin(deg(x)))}, {2*exp(-0.2*x)* (2.5*cos(deg(x)) +1.5*sin(deg(x)))}); 

	\node[label={270:{$(\bar{x},\bar{y})$}},circle,fill,inner sep=0.7pt] at (axis cs:2.307,-1.341) {};
	%
	\addplot [->-=0.33,->-=0.66,domain=0:1.7894,samples=100,blue] ({exp(-0.2*x)*(-2*cos(deg(x)) + 0.5*sin(deg(x)))}, {exp(-0.2*x)* (2.5*cos(deg(x)) +1.5*sin(deg(x)))}); 
	\node[label={270:{$(x_0,y_0)$}},circle,fill,inner sep=0.7pt] at (axis cs:-2,2.5) {};
	
	\addplot [->-=0.5,domain=0.6445:1.289, samples=5, red] (x, x); 

	\node[circle,fill,inner sep=0.5pt] at (axis cs:0.6445,0.6445) {};
	\node[circle,fill,inner sep=0.5pt] at (axis cs:1.289,1.289) {};

    \end{axis}

    \end{tikzpicture}
\end{center}
  \caption{The non-smooth trajectory $\xx_s$ constructed in \eqref{eq:non-smoothtraj}. A few orbits of
  the dynamic with $\zeta = 0$ are drawn in light blue and a few orbits with $\zeta = 5$ are drawn in light red.}
  \label{fig:control}
\end{figure}

Setting $\bar{\xx}:=(\bar{x},\bar{y})\neq0$, define
\begin{align} 
\bar{t}:=\inf\left\{t\geq 0: \e^{-A_0 t}\bar{\xx}  \in \ell_u\right\}<\infty\;.
\end{align}
Now, given $\xx_0=(x_0,y_0)$,
consider the forward solution to $\dot{\xx}=A_0\xx$ with initial condition 
$\xx(0)=\xx_0$. It is clear from \eqref{eq:eigenvaluesA1} that $\e^{A_0t}\xx_0\to0$ as $t\to\infty$ and that $\e^{A_0t}\xx_0$ intersects
$\ell_u$ infinitely many times. We define
\begin{equs}
t_0:=\inf\Big\{t\geq 0:\, \e^{A_0 t}\xx_0 \in \ell_u, \quad \e^{A_0 t}\xx_0 &\cdot \e^{-A_0 \bar{t}}\bar{\xx}>0,  \\
&\big|\e^{A_0 t}\xx_0\big|\leq \big|\e^{-A_0 \bar{t}}\bar{\xx}\big|\Big\}<\infty.
\end{equs}
The definition of the time $t_0$ makes sure that not only $ \e^{A_0 t_0}\xx_0 \in \ell_u$, but also that it is on the same
ray as $ \e^{-A_0 \bar{t}}\bar{\xx}$ and that enough time as passed so that $ \e^{A_0 t_0}\xx_0$ is closer to the origin than $ \e^{-A_0 \bar{t}}\bar{\xx}$.
Finally, we define $t_1\geq0$ to be the time such that
\begin{align}
\e^{A_1 t_1} \e^{A_0 t_0}\xx_0= \e^{-A_0 \bar{t}}\bar{\xx}.
\end{align}
Therefore, defining  $T:=t_0+t_1+\bar{t}$, we find that the piecewise continuous curve
\begin{equ}\label{eq:non-smoothtraj}
\xx_s(t)=\begin{cases}
\e^{A_0 t}\xx_0, \quad &t\in[0,t_0],\\
\e^{A_1 (t-t_0)}\e^{A_0 t_0}\xx_0, \quad &t\in[t_0,t_0+t_1],\\
\e^{A_0 (t-t_0-t_1)}\e^{A_1 t_1}\e^{A_0 t_0}\xx_0, \quad &t\in[t_0+t_1,T],
\end{cases}
\end{equ}
is such that $\xx_s(0)=\xx_0$ and $\xx_s(T)=\bar{\xx}$. 

\medskip

\noindent\emph{Step 2.}
We proceed with a suitable smoothing of the trajectory constructed in step 1. For a small $\delta\in(0,1)$, define a smooth function $\zeta_\delta$ such that $\zeta_\delta(0)=z_0$, $\zeta_\delta(T)=\bar{z}$ and
\begin{align}\label{eq:zetadelta}
\zeta_\delta(t)=
\begin{cases}
0, \quad &t\in [\delta,t_0],\\
5, \quad &t\in [t_0+\delta,t_0+t_1-\delta],\\
0, \quad &t\in[t_0+t_1,T-\delta].
\end{cases}
\end{align}
We now compare the dynamics of the ODEs generated by the matrices
\begin{align}
A(t)=
\begin{pmatrix}
0& 1\\
\zeta_\delta(t)-2  & -2
\end{pmatrix}
\end{align}
and 
\begin{align}
B(t)=
\begin{cases}
A_0, \quad &t\in [0,t_0],\\
A_1, \quad &t\in [t_0,t_0+t_1],\\
A_0, \quad &t\in[t_0+t_1,T].
\end{cases}
\end{align}
For $t\geq\tau\geq 0$, we denote by $\Phi^A_{t,\tau}, \Phi^B_{t,\tau}:\R^2\to\R^2$ the respective solution operators. 
From the definition above, it is clear that
$\Phi^B_{t,0}(\xx_0)=\xx_s(t)$.
Moreover, the underlying dynamics is the same for the two systems, up to a time-interval of size $4\delta$.
Since
\begin{align}\label{eq:linearestimate}
\sup_{t\in [0,T]}\left[\|A(t)\|+\|B(t)\| +\left|\xx_s(t)\right| \right]\leq C,
\end{align}
where $C$ is independent of $\delta$,
we deduce that 
\begin{align}\label{eq:deltaerror}
\sup_{t\in [0,T]}\left|\Phi^A_{t,0}(\xx_0) -\xx_s(t)\right| \leq C \delta,
\end{align}
for possibly a different $C$, independent of $\delta$.

\medskip

\noindent\emph{Step 3.} To conclude the proof,
fix $\eps>0$ and choose $\delta$ in \eqref{eq:deltaerror} so that $C \delta<\eps$. 
By writing $(x_c(t),y_c(t),z_c(t))=(\Phi^A_{t,0}(\xx_0), \zeta_\delta(t))$ as in the previous step,
we define
\begin{align}
h=\dot{z}_c+\gamma (z_c - z_\star) + x_c(x_c+ \eta y_c).
\end{align}
It is then clear that $(x_c(t),y_c(t),z_c(t))$ is the unique solution to \eqref{e:LorenzControl} with initial
datum $(x_0,y_0,z_0)$. Moreover, $z_c(T)=\bar{z}$ and by \eqref{eq:deltaerror}  and our choice of $\delta$, 
the proof is over.
\end{proof}

\begin{remark}
A suitable modification of the above proof implies that the time $T$ can be made arbitrarily small thanks to the fact
that the eigenvalues of the $2\times2$ system can be completely controlled through the $z$ variable. However, we
will not need this enhancement for our purposes.
\end{remark}

\section{Construction of a Lyapunov function}
\label{sec:Lyapunov}

We now proceed with the analysis of the ``linearised'' model
\begin{equ}[e:ThetaZ]
\dot \theta = 1 - z \sin^2 (\theta)\;,\qquad
\dot{z}=- \gamma  (z-z_\star) + \alpha\,\xi\;,
\end{equ}
obtained from \eqref{e:equationR}-\eqref{e:defOU}. 
Here, we are taking advantage of the fact that $r$ does not appear in the right-hand side of \eqref{e:equationR}-\eqref{e:defOU}, and it is
therefore completely determined once $(\theta,z)$ are.
Note that by Propositions~\ref{prop:Hormander} and~\ref{prop:control},
\eqref{e:ThetaZ} is strong Feller and topologically irreducible. Since the
state space of $\theta$ is compact and $z$ is a simple Ornstein--Uhlenbeck process,
it admits Lyapunov functions with compact sublevel sets (the map $(\theta,z) \mapsto |z|^2$ will do) and therefore has a unique invariant measure $\mu_\alpha$.

The averaged behaviour of the right-hand side of the $r$ equation in \eqref{e:equationR} turns out to characterise the invariant measure of the 
full Lorenz system \eqref{e:LorenzNew},
as the following theorem shows.
\begin{theorem}\label{theo:Lyapunov}
Let $\mu_\alpha$ be the invariant measure for \eqref{e:ThetaZ}, and set
\begin{equ}
\lambda_\alpha \eqdef -1 + \frac12\int_{S^1 \times \R} z\sin(2\theta)\, \mu_\alpha(\dd\theta,\dd z)\;.
\end{equ} 
Then, the Lorenz system \eqref{e:LorenzNew} admits a nontrivial invariant measure if $\lambda_\alpha > 0$ and 
admits no such measure if $\lambda_\alpha < 0$.
\end{theorem}
The proof of this theorem, which is given in Section~\ref{sec:actualProof} below, 
is based on the construction of a Lyapunov functional for the  Lorenz system \eqref{e:LorenzNew}
that blows up as $|x|^2+|y|^2\to 0$ (in the unstable case $\lambda_\alpha > 0$) and at infinity. 
This follows by a perturbative analysis from a similar analysis of the 
linearised system, which we carry out in the next section.

\subsection{Lyapunov functional for the linearised system}
To show the existence of a Lyapunov functional for the linearised system \eqref{e:ThetaZ} we need to analyse 
the regularity properties of the solution $g_\alpha$ to the problem
\begin{equ}[e:galphaPDE]
\CL_0 g_\alpha = \lambda_\alpha + 1 - {z\over 2} \sin(2\theta)\;,
\qquad \mu_\alpha(g_\alpha) = 0\;.
\end{equ}
Notice that $g_\alpha$ does not depend on the $r$ variable. The following result contains
all the properties needed later.

\begin{lemma}\label{lem:bounds}
For every $\eps > 0$, there exists a unique solution $g_\alpha$ to \eqref{e:galphaPDE} 
in $L^2(S^1 \times \R, \e^{-\eps z^2}d\theta\,dz)$.
Furthermore, $g_\alpha$ is smooth and there exists $c_\alpha>0$ (depending on $\eps$) such that
\begin{equ}[e:galphabounds]
|g_\alpha| + |\d_z g_\alpha| \le c_\alpha \e^{\eps z^2/2}\;, 
\end{equ}
holds uniformly over $(\theta, z)$.
\end{lemma}

\begin{proof}
Smoothness of $g_\alpha$ (once we know that it exists) follows from
Hörmander's theorem \cite{Hor67AM}, noting that 
Proposition~\ref{prop:Hormander} guarantees
that its assumptions are verified (Hörmander's condition for $\CL$ is equivalent to
that of $\CL_1$, which in turn implies it for $\CL_0$).
Note first that we can assume without loss of generality that $\eps$ is small enough since 
\eqref{e:galphabounds} then holds automatically for larger values of $\eps$.
Setting
\begin{equ}
G(\theta,z):= \lambda_\alpha + 1 - {z \over 2} \sin(2\theta)\;,\qquad 
\phi(\theta,z):= g_\alpha(\theta,z) \e^{-\eps z^2/2}\;,
\end{equ}
we see that \eqref{e:galphaPDE} is equivalent to the equation
\begin{equ}\label{e:Kequation}
K \phi = \Phi\;, \qquad \text{where}\quad  \Phi(\theta,z)= \e^{-\eps z^2/2} G(\theta,z)
\end{equ}
and
\begin{equ}
K = \CL_0 + \alpha^2\eps z\de_z - \frac12\left[\eps \left(2\gamma-\alpha^2\eps \right)z^2-2\gamma\eps z_\star z-\alpha^2\eps \right]\;.
\end{equ}
This operator belongs to the class $\mathcal{K}_0$ as defined in \cite[Def.~2.2]{EH03}, on the space $L^2(S^1\times\R)$. 
Setting
\begin{equ}
b(\theta,z) = 1- z \sin^2(\theta)\;,
\end{equ}
integration by parts shows that, for any $\psi \in \CC_0^\infty$,
\begin{equs}
-\int \psi\,K\psi&=\frac12\int \de_\theta b|\psi|^2+\frac{\alpha^2}{2}\|\de_z\psi\|^2-\frac{\gamma}{2}\|\psi\|^2+\frac{\eps}{2} \left(2\gamma-\alpha^2\eps \right)\|z \psi\|^2\\
&\quad -\gamma\eps z_\star\int z|\psi|^2\;.
\end{equs}
In particular, for any positive $\eps$ such  that
$\eps\leq \gamma / \alpha^2$,
and using the bound
$|\de_\theta b|\lesssim |z|$,
we infer by Cauchy--Schwarz that
\begin{equ}[e:niceBound]
\|\de_z\psi\|^2+\|z \psi\|^2 \lesssim \|K\psi\|^2+ \|\psi\|^2 \;,
\end{equ}
which can be extended to all $\psi \in \CD(K)$ by a simple approximation argument.
By \cite[Cor~4.2]{EH03}, it follows that $K$ has compact resolvent. 

Let now $\lambda > 0$ be in the resolvent set for $K$ (such a $\lambda$ exists since
the spectrum is discrete) and write $R_\lambda$ for the resolvent so that, since
$R_\lambda$ is a bijection between $L^2$ and $\CD(K)$, \eqref{e:Kequation} is equivalent to 
\begin{equ}
(\id +\lambda R_\lambda)\phi = R_\lambda \Phi\;.
\end{equ}
Note now that if $f \in \ker K^*$, then $\hat f := \e^{-\eps z^2/2} f$ satisfies
$\CL_0^* \hat f = 0$. Since $\hat f \in L^1$, we conclude as in the proof of \cite[Prop.~3.6]{EPR} that 
 $|\hat f|$ is the density of an invariant measure for \eqref{e:ThetaZ}. Since the latter
is unique and since $\e^{-\eps z^2/2} \in \ker K$ (so 
that $\dim \ker K = \dim \ker K^* \ge 1$), 
we conclude that $\ker K^*$ is one-dimensional and spanned by the element $f \in L^2$
such that 
\begin{equ}
\mu_\alpha(\dd u) = \e^{-\eps z^2/2}f(u)\,\dd u\;.
\end{equ}
Since $\mu_\alpha(G) = 0$ by definition of $\lambda_\alpha$ and therefore
$\scal{f, \Phi} = 0$ so that $\Phi \in (\ker K^*)^\perp$, we deduce  
 from Fredholm's alternative \cite[Thm~4.25]{Rudin} that \eqref{e:Kequation} admits a unique solution
 $\phi$ with the additional property that the corresponding 
function $g_\alpha$ is centred with respect to $\mu_\alpha$.

Note that it immediately follows from \eqref{e:niceBound} that
\begin{align}\label{eq:L2est}
\|\de_z\phi\|^2+\|(1+z) \phi\|^2 < \infty\;.
\end{align}
To derive similar bounds on higher derivatives, we simply take derivatives of 
\eqref{e:Kequation}, making sure that the
corresponding commutators are well-behaved. We have
\begin{align*}
[\de_\theta,K]&=\de_\theta b\,\de_\theta\;,\\
[\de_{\theta}^2,K]&=\left[\de_\theta,[\de_\theta,K]\right]+2[\de_\theta,K]\de_\theta=\de_{\theta}^2b\,\de_\theta +2\de_\theta b\,\de_\theta^2\;,\\
[\de_{\theta}^3,K]&=\left[\de_\theta,\left[\de_\theta,[\de_\theta,K]\right]\right]+ 3\left[\de_\theta,[\de_\theta,K]\right]\de_\theta+3[\de_\theta,K]\de_{\theta}^2\\
&=\de_{\theta}^3b\,\de_\theta+3\de_{\theta}^2b\,\de_\theta^2 +3\de_\theta b\,\de_{\theta}^3\;,
\end{align*}
and, noting
that $\de_z b$ is constant in $z$, also
\begin{align*}
[\de_z,K]&=\de_z b\, \de_{\theta}-\left(\gamma-\alpha^2\eps\right)\de_z-\eps \left(2\gamma-\alpha^2\eps \right)z+\gamma\eps z_\star\;,\\
[\de_{z}^2,K]
&=-\eps \left(2\gamma-\alpha^2\eps \right)+2[\de_z,K]\de_z\;.
\end{align*}
In light of \eqref{eq:L2est}, we have that $[\de_\theta,K]\phi$ is in $L^2$, hence \eqref{eq:L2est} holds also for $\de_\theta \phi$. Thus $ [\de_z,K]\phi\in L^2$ as well and the same
conclusion for $\de_z\phi$ follows. 
Proceeding iteratively, we deduce that $\phi \in H^3 (S^1\times \R)$, and the conclusion follows from the embedding $H^3 \subset W^{1,\infty}$. 
\end{proof}

The construction and properties of the Lyapunov functional for the full linearised system \eqref{e:equationR}-\eqref{e:defOU}
is contained in the following proposition.

\begin{proposition}\label{prop:boundV0}
Let
\begin{equ}\label{e:epsalpha}
\eps_\alpha =\frac{\gamma}{2\Gamma} \wedge {\beta \nu^2\sigma^3\chi^4 \over 16\alpha^2}\;,\qquad \Gamma = \alpha^2+2\gamma z_\star^2\;,
\end{equ}
and let $V_0$ be given by 
\begin{equ}\label{e:defV_0}
V_0(r,\theta,z) = \e^{-\kappa r} \left(1 - \kappa g_\alpha(\theta,z) + \delta \e^{\eps_\alpha z^2}\right)\;,
\end{equ}
for some constants $\kappa \in \R$ and $\delta > 0$. Then, for every $\alpha \in \R$ such that 
$\lambda_\alpha \neq 0$ there exists a choice of $\delta,\kappa$ with $\sgn \kappa = \sgn \lambda_\alpha$ such that 
\begin{equ}[e:boundV0]
 \e^{-\kappa r} \left(1 + \delta \e^{\eps_\alpha z^2}\right)
 \le 2V_0 \le 3 \e^{-\kappa r} \left(1 + \delta \e^{\eps_\alpha z^2}\right)\;,
\end{equ} 
and such that furthermore 
$
\CL_1 V_0 \le -d \left(1+z^2\right) V_0
$
for some constant $d>0$.
\end{proposition}

\begin{proof}
From \eqref{e:galphaPDE} and the definition of the generators \eqref{e:genL1}--\eqref{e:genL0}, it follows that
\begin{equ}
\CL_0g_\alpha=\lambda_\alpha-\CL_1r.
\end{equ}
A simple calculation then shows that
\begin{equs}
\CL_1 V_0 &= \e^{-\kappa r} \Big[(-\kappa \CL_1 r)\left(1 - \kappa g_\alpha + \delta \e^{\eps_\alpha z^2}\right) - \kappa \CL_0 g_\alpha\\ 
&\quad + \delta \left(\alpha^2 \eps_\alpha +2\gamma\eps_\alpha z_\star z+ 2\eps_\alpha\left(\alpha^2 \eps_\alpha - \gamma\right)z^2\right)\e^{\eps_\alpha z^2}\Big] \\
&= \e^{-\kappa r} \Big[\kappa \left(\CL_1 r\right)\left( \kappa g_\alpha - \delta \e^{\eps_\alpha z^2}\right) - \kappa\lambda_\alpha\\ 
&\quad + \delta \left(\alpha^2 \eps_\alpha +2\gamma\eps_\alpha z_\star z+ 2\eps_\alpha\left(\alpha^2 \eps_\alpha - \gamma\right)z^2\right)\e^{\eps_\alpha z^2}\Big] \\
&\le \e^{-\kappa r} \left[\kappa \left(\CL_1 r\right)\left( \kappa g_\alpha - \delta \e^{\eps_\alpha z^2}\right) - \kappa\lambda_\alpha
+ \delta \eps_\alpha \left(\Gamma  - \frac{\gamma z^2}{2} \right)\e^{\eps_\alpha z^2}\right]\;,
\end{equs}
where we used the fact that 
$\eps_\alpha  \le {\gamma\over 2\alpha^2}$
in order to obtain the last inequality.
Note now that since the function $u \mapsto \e^{u^2}(3-u^2)$ is bounded from above by $8$, we deduce that 
\begin{equ}[e:boundGamma]
\Gamma  - \frac{\gamma z^2}{2} 
\le 4 \Gamma \e^{- \frac{\gamma z^2}{2\Gamma}} - {2\Gamma + \gamma z^2\over 4}
\le 4 \Gamma \e^{- \eps_\alpha z^2} - {2\Gamma + \gamma z^2\over 4}\;,
\end{equ}
so that
\begin{equs}
\CL_1 V_0 \le \e^{-\kappa r} 
&\Big[\kappa \left(\CL_1 r\right)\left( \kappa g_\alpha - \delta \e^{\eps_\alpha z^2}\right)
+ 4\delta \eps_\alpha \Gamma - \kappa \lambda_\alpha 
- {\delta \eps_\alpha \over 4} \left(2\Gamma + \gamma z^2\right)\e^{\eps_\alpha z^2}\Big]\;.
\end{equs}
We now use the fact that, by Lemma~\ref{lem:bounds}, there exists a constant $c_\alpha$ such that 
\begin{equ}
|g_\alpha| \le c_\alpha \e^{\eps_\alpha z^2}\;,
\end{equ}
so that, since  furthermore
\begin{equ}
|\CL_1 r| \le 2+z^2\;,
\end{equ}
we have
\begin{equs}
\CL_1 V_0 \le \e^{-\kappa r} 
&\Big[|\kappa|  \left(2+z^2\right)\left(|\kappa| c_\alpha + \delta\right) \e^{\eps_\alpha z^2}
+ 4\delta \eps_\alpha\Gamma - \kappa \lambda_\alpha \\
&\quad- {\delta \eps_\alpha \over 4} \left(2\Gamma + \gamma z^2\right)\e^{\eps_\alpha z^2}\Big]\;.
\end{equs}
We now make the choices 
\begin{equ}
\delta = |\kappa|^{3/2}\;, \qquad |\kappa| \le  
{\lambda_\alpha^2 \over 64\Gamma^2} \wedge 
{\eps_\alpha (\gamma\wedge\Gamma) \over 16} \wedge 
\left({\eps_\alpha (\gamma\wedge\Gamma) \over 16 c_\alpha}\right)^2 \;, 
\end{equ}
(as well as $\sgn \kappa = \sgn \lambda_\alpha$ as in the statement) so that we obtain the bound
\begin{equs}
\CL_1 V_0 
&\le \e^{-\kappa r} \left(- {\kappa \lambda_\alpha\over 2} 
- {\delta \eps_\alpha \over 8} \left(2\Gamma + \gamma z^2\right) \e^{\eps_\alpha z^2}\right)\\
&\le \e^{-\kappa r} \left(- {\kappa \lambda_\alpha\over 2} 
- {\delta \eps_\alpha \over 8} \left(2\alpha^2 + \gamma z^2\right) \e^{\eps_\alpha z^2}\right)\\
&\le  -\left(\frac{\kappa \lambda_\alpha}{2} \wedge  { \gamma \over 8}\right)   \e^{-\kappa r} \left( 1 
+ \delta \left(1+\eps_\alpha z^2\right) \e^{\eps_\alpha z^2}\right)\;.
\end{equs}
If we furthermore impose $|\kappa| \le 1/ c_\alpha^4$, then it follows from \eqref{e:galphabounds} that
\begin{equ}
|\kappa||g_\alpha| \le |\kappa|c_\alpha \e^{\eps_\alpha z^2/2} \le {1\over 2} \left(1 + \delta \e^{\eps_\alpha z^2}\right)\;,
\end{equ}
so that we do indeed have the bound \eqref{e:boundV0}. In particular, since $\delta\leq 1$ this implies
\begin{equ}
2\left(1+\delta\eps_\alpha z^2\right)V_0
\leq 3 \e^{-\kappa r} \left(1+ \delta\left(1+2\eps_\alpha z^2\right)\e^{\eps_\alpha z^2}\right)\;.
\end{equ}
This finally leads to the bound
\begin{equ}
\CL_1 V_0 \le -\frac16\left(\kappa \lambda_\alpha \wedge  { \gamma \over 4}\right) \left(1+\delta\eps_\alpha z^2\right)V_0\;,
\end{equ}
as required.
\end{proof}

\subsection{Lyapunov functional for the stochastic Lorenz system}
\label{sec:actualProof}

We also need the following standard result which
immediately follows from the Dambis--Dubins--Schwarz representation of continuous
martingales as a time-changed Brownian motion \cite[Thm~V.1.6]{RevuzYor}.

\begin{lemma}\label{lem:divergence}
Let $X_t  = X_0 + A_t + M_t$ be a continuous semimartingale with $A_0 = M_0 = 0$
such that there exists a constant $\kappa$ for which $A_t \le -\kappa \scal{M}_t$.
Then, provided that $\lim_{t \to \infty} \scal{M}_t = \infty$ almost surely, one has
$\lim_{t\to \infty} X_t = -\infty$ almost surely. \qed
\end{lemma}

We are now ready to give the proof of Theorem~\ref{theo:Lyapunov}.

\begin{proof}[Proof of Theorem~\ref{theo:Lyapunov}]
We treat the cases $\lambda_\alpha > 0$ and $\lambda_\alpha < 0$ separately.
In the case $\lambda_\alpha > 0$, it suffices to find a function $V \colon \big(\R^2 \setminus \{(0,0)\}\big)\times\R \to \R_+$
with compact level sets and such that $\CL V \le K - c V$ for some positive constants $c$ and $K$.

For this, we first go back to the original formulation \eqref{e:Lorenz}, 
write $U = (X,Y,Z)$ and define the norm
\begin{equ}
|U|^2 =  X^2 + Y^2 + (Z-\sigma-\rho)^2\;.
\end{equ}
For any $\bar{c}>0$ to be fixed, we define the functional
\begin{equ}\label{eq:defV1tilde}
\widetilde{V}_1(X,Y,Z) = \exp\left(\bar{c} |U|^2\right).
\end{equ}
Applying the generator $\widetilde{\CL}$ of \eqref{e:Lorenz}, we see that $\widetilde{V}_1$ satisfies the identity
\begin{equ}
\widetilde{\CL} \widetilde{V}_1=2\bar{c} \widetilde{V}_1 \left(\frac{\hat{\alpha}^2}{2}-\sigma X^2-Y^2-\left(\beta-\bar{c}\hat{\alpha}^2\right) (Z-\sigma-\rho)^2-\beta(\sigma+\rho)(Z-\sigma-\rho)\right).
\end{equ}
Setting 
$\bar{c}= {\beta \over 2\hat{\alpha}^2}$,
we get
\begin{equs}
\widetilde{\CL}  \widetilde{V}_1&=2\bar{c} \widetilde{V}_1 \left(\frac{\hat{\alpha}^2}{2}-\sigma X^2-Y^2-\frac{\beta}{2} (Z-\sigma-\rho)^2-\beta(\sigma+\rho)(Z-\sigma-\rho)\right)\\
&\leq 2\bar{c} \widetilde{V}_1 \left(\frac{\hat{\alpha}^2}{2}+\beta(\sigma+\rho)^2-\sigma X^2-Y^2-\frac{\beta}{4} (Z-\sigma-\rho)^2\right),
\end{equs}
which implies that
\begin{equ}[e:boundtilV1]
\widetilde{\CL}  \widetilde{V}_1 
\le \widetilde{K} - d_{\hat{\alpha}}\left(1 + |U|^2\right)  \widetilde{V}_1\;,
\end{equ}
where $d_{\hat{\alpha}}\sim \hat{\alpha}^{-2}$ and $\widetilde{K}>0$ is independent of $\hat{\alpha}$. In the variables $(x,y,z)$ of system \eqref{e:LorenzNew}, it then follows from \eqref{e:boundtilV1} that
the functional
\begin{equ}\label{eq:V1def}
V_1(x,y,z)=\widetilde{V}_1(X,Y,Z)
\end{equ}
satisfies
\begin{equ}[e:boundV1]
\CL  V_1 
\le K - d_\alpha\left(1 + x^2+y^2+z^2\right)  V_1\;,
\end{equ}
for different constants $K$ and $d_\alpha$, explicitly computable from \eqref{e:newVar}.
The bound \eqref{e:boundV1} suggests the natural choice
\begin{equ}
V = V_0 + V_1\;,
\end{equ}
leading to the identity
\begin{equ}\label{e:ident}
\CL V = \CL_1 V_0 - x(x+\eta y) \d_z V_0 + \CL V_1\;.
\end{equ}
Combining \eqref{e:ident} with Proposition~\ref{prop:boundV0} and the bound \eqref{e:boundV1} yields the estimate
\begin{equ}
\CL V \le K - d_\alpha\left(1 + z^2\right) V - x(x+\eta y) \d_z V_0\;,
\end{equ}
for possibly different constants $d_\alpha$ and $K$.

We then note that, as a consequence of Lemma~\ref{lem:bounds}, there exists a constant $c$ such that 
\begin{equ}[e:boundDerV0]
|\d_z V_0| \le c \left(1+|z|\right) V_0\;.
\end{equ} 
It is then immediate that in the region $3 x^2+\eta^2y^2 \le d_\alpha/(2c)$, we have the bound
\begin{equ}
|x(x+\eta y) \d_z V_0| \le {d_\alpha \over 2}\left(1 + z^2\right) V\;.
\end{equ}
On the other hand, in the region $3 x^2+\eta^2 y^2 \ge d_\alpha/(2c)$, since $\lambda_\alpha>0$ we have that
\begin{equ}
V_0 \lesssim\e^{\eps_\alpha z^2} \lesssim\e^{2\eps_\alpha (z-z_\star)^2}.
\end{equ} 
Now, comparing the above upper bound with the definitions of $V_1$ in \eqref{eq:V1def} and 
of $\widetilde{V}_1$  in \eqref{eq:defV1tilde} via the change of variables \eqref{e:newVar}, it is not hard to see that if (compare with
\eqref{e:epsalpha})
\begin{equ}
\eps_\alpha\leq \frac{\bar{c}\chi^4\sigma^2}{4}= \frac{\beta \chi^4\sigma^2}{8\hat{\alpha}^2},
\end{equ}
then there holds
\begin{equation}
V_0 \lesssim  \sqrt{V_1}\;.
\end{equation} 
In particular, we can find $K>0$ such that 
\begin{equ}
|x(x+\eta y) \d_z V_0| \le K + {d_\alpha\over 2} V\;.
\end{equ}
Combining these bounds does indeed yield 
\begin{equ}
\CL V \le K - {d_\alpha \over 2} V \;,
\end{equ}
as required.

Regarding the case $\lambda_\alpha < 0$, it suffices to show that, for any realisation of the process,
$\lim_{t \to \infty} |x^2_t + y^2_t| = 0$ almost surely, which follows in particular
if we can show that $\lim_{t \to \infty} V_0(t) = 0$. (Note that in this
case $\kappa < 0$ in the definition of $V_0$!) By Proposition~\ref{prop:boundV0}, we have the bound
\begin{equ}
\CL \log V_0 = {\CL V_0 \over V_0} - {\alpha^2\over 2} \left({\d_z V_0 \over V_0}\right)^2
\le -d\left(1+z^2\right) + x(x+ \eta y){\d_z V_0 \over V_0}\;.
\end{equ}
Note furthermore that since $\kappa<0$, it follows by \eqref{e:boundV0} that $V_0$ is bounded below by~$1$. 
Hence, \eqref{e:boundDerV0}  entails
\begin{equ}
|\d_z \log V_0|^2 \le \tilde c(1+z^2)\;.
\end{equ}
It follows that there exists a constant $\delta$ such that, as long as $x^2_t + y^2_t \le \delta$,
one has
\begin{equ}
\dd\log V_0 \le -d\left(1+z^2\right)\, \dd t + \sqrt{\tilde c\left(1+z^2\right)}\,\dd W_t\;,
\end{equ}
for some Wiener process $W_t$. It then follows from Lemma~\ref{lem:divergence} that
there exists some $\eps_0 > 0$ such that, uniformly over all initial conditions with
$x^2_0 + y^2_0 = \delta/2$, one has $\lim_{t\to \infty} \log V_0(t) = -\infty$
and $\sup_{t > 0} \big(x^2_t + y^2_t\big) \le \delta$ with probability at least $\eps_0$.

It furthermore follows from Proposition~\ref{prop:control} and the fact that $V_1$ is a 
global Lyapunov function for our system that, for any fixed initial condition
with $x^2 + y^2 > \delta/2$, the stopping time $\tau= \inf\{t > 0\,:\, x^2 + y^2 = \delta/2\}$ is 
almost surely finite and admits exponential moments.
If we now split the trajectory into excursions between the cylinders
$\{x^2 + y^2 = \delta\}$ and $\{x^2 + y^2 = \delta/2\}$, a simple renewal argument
shows that one has indeed $\lim_{t\to \infty} \log V_0(t) = -\infty$ almost surely, as required.
\end{proof}

\section[Behaviour of the angular motion for small and large \texorpdfstring{$\alpha$}{alpha}]{Behaviour of the angular motion for small and large \texorpdfstring{$\boldsymbol{\alpha}$}{alpha}}
\label{sec:largealpha}

Recall that the system under consideration is given by
\begin{equ}[e:equationTheta]
\dot \theta = 1 - z \sin^2 (\theta)\;,\qquad
\dot{z}=- \gamma  (z-z_\star) + \alpha\,\xi\;,
\end{equ}
and write as before $\mu_\alpha$ for its invariant measure (we consider $\gamma$ and $z_\star$ as fixed). 
Then, one can write $\lambda_\alpha$ as
\begin{equ}[e:defLalpha]
\lambda_\alpha = -1 + {1\over 2}\int z\sin(2\theta)\,\mu_\alpha(\dd\theta, \dd z) \;.
\end{equ}

\begin{remark}
Since both \eqref{e:equationTheta} and \eqref{e:defLalpha} are invariant under $\theta \mapsto \theta + \pi$,
we can view $\theta$ as an element of the real projective line 
$\R\P^1 \equiv [-{\pi\over2}, {\pi\over2}] / \{{\pi\over2} , - {\pi\over2}\}$. 
For fixed $z > 1$, the equation for $\theta$ in \eqref{e:equationTheta} then admits 
 exactly two fixed points $\theta_\pm$ with
\begin{equ}[e:defFP]
\sin\theta_\pm = \pm{1\over \sqrt z} \;,\qquad \Big|\theta_\pm - {\pm 1\over \sqrt z}\Big|
\le {1 \over z^{3/2}}\;.
\end{equ}
The fixed point $\theta_+$ is stable, while $\theta_-$ is unstable with a saddle-node bifurcation at $z=1$
when $\theta_\pm = {\pi\over 2}$.
\end{remark}

In this section, we exhibit the precise asymptotic behaviour of $\lambda_\alpha$ for small and
for large values of $\alpha$. Since the map $\alpha \mapsto \lambda_\alpha$ is smooth, this
immediately yields Theorem~\ref{theo:main} when combined with Theorems~\ref{theo:basic} and~\ref{theo:Lyapunov}.
\begin{theorem}\label{theo:behaviourLyap}
One has
\begin{equ}\label{e:lambalph}
\lim_{\alpha \to \infty} \alpha^{-1/2} \lambda_\alpha = {\Gamma({3\over4})\over 2\gamma^{1\over 4}\pi^{1\over 2}}
 \;,\qquad 
\lim_{\alpha \to 0} \lambda_\alpha = 
\left\{\begin{array}{cl}
	\sqrt{z_\star-1}-1 & \text{if $z_\star > 1$,} \\
	-1 & \text{otherwise.}
\end{array}\right.\;.
\end{equ}
\end{theorem}

\begin{remark}
Our proof actually shows that 
\begin{equ}[e:largeAlpha]
\biggl|\lambda_\alpha - {\alpha^{1\over2} \Gamma({3\over4})\over 2\gamma^{1\over 4}\pi^{1\over 2}}\biggr|
\lesssim \alpha^{{1\over 3} + \kappa}\;,\qquad \text{as $\alpha \to \infty$,}
\end{equ}
(for any $\kappa > 0$) while we have
\begin{equ}[e:smallAlpha]
|\lambda_\alpha - \lambda_0| \lesssim 
\left\{\begin{array}{cl}
	\alpha^{1/4} & \text{if $z_\star = 1$,} \\
	\alpha^{3/4} & \text{otherwise.}
\end{array}\right.
\end{equ}
Heuristics suggest that \eqref{e:smallAlpha} should hold with 
exponents $1$ and $2$ respectively, which we expect to be optimal. It is however not clear to us what the 
optimal exponent in \eqref{e:largeAlpha} should be.
\end{remark}

\begin{proof}
We start with the case $\alpha \to \infty$ since this is the harder one. Setting
\begin{equ}[e:defFF]
F(\theta,z) = -1 + {z\over 2} \sin(2\theta)\;,\qquad 
F_\infty(\theta,z) = \sqrt{z}\one_{z > 0}\;,
\end{equ}
the claim follows if we can show that 
\begin{equ}[e:wantedBoundDiff]
\lim_{\alpha \to \infty} \alpha^{-1/2}\Big|\int \big(F-F_\infty\big)\,\dd \mu_\alpha\Big| = 0\;,
\end{equ}
where $ \mu_\alpha$ denotes the invariant measure for the system \eqref{e:equationTheta}.
This is because $z \sim \CN(z_\star, \alpha^2/(2\gamma))$ under $\mu_\alpha$ and, 
for $X \sim \CN(0,1)$, one has 
\begin{equ}
\E \big(\one_{X > 0}\sqrt X\big)
= {1\over \sqrt{2\pi}}\int_0^\infty \e^{-{X^2\over2}} \sqrt X\, \dd X
= {1\over 2^{3\over 4}\pi^{1\over 2}}\int_0^\infty \e^{-t}t^{-{1\over4}}\, \dd t\;.
\end{equ}
In order to show that \eqref{e:wantedBoundDiff} holds, we consider the process given by \eqref{e:equationTheta} and
we define the following increasing sequence of stopping times. We set $\tau_0= 0$
and then inductively
\begin{equ}[e:deftaun]
\tau_{n+1} = \inf\{t > \tau_n \,:\, z(t) \not\in \CZ(z(\tau_n))\}\;,
\end{equ}
where the regions $\CZ(z) \subset \R$ are defined by
\begin{equ}[e:defZz]
\CZ(z) = 
\left\{\begin{array}{cl}
	[-1,1] & \text{if $|z| \le 1/2$,} \\
	\{c z\,:\, c \in [1/2,2]\} & \text{otherwise.}
\end{array}\right.
\end{equ}
Our definitions guarantee that $|z(\tau_{n+1})-z(\tau_{n})| \ge 1/4$, so that 
since $z$ is a simple Ornstein-Uhlenbeck process (and therefore has continuous sample paths), 
the stopping times $\tau_n$ necessarily increase to $+\infty$.
Using the explicit representation of the Ornstein--Uhlenbeck process $z(t)$, it
is furthermore straightforward to see that for every $p \ge 1$
there exists a constant $C$ such that the bound
\begin{equ}[e:bound excursionTime]
\E \big((\tau_{n+1} - \tau_n)^p\,|\, \CF_{\tau_n}\big) \le C\;,
\end{equ} 
holds almost surely, uniformly over $n$ (see Lemma~\ref{lem:stopTime} below
for a much more precise upper and lower bound). 

Writing $\CX = \R\P^1 \times \R$ for the state space of the Markov process \eqref{e:equationTheta}, we
define a space of excursions $\hat \CX = \R \times \CC(\R_+ , \CX) / \sim$, where
we set
\begin{equ}
(\tau,u) \sim (\bar \tau,\bar u) \quad\text{iff}\quad \bar \tau = \tau \quad\text{and}\quad u(t) = \bar u(t)\; \forall t \le \tau\;,
\end{equ}
and we define a sequence of excursions $E_n \in \hat \CX$ by
\begin{equ}[e:defEn]
E_n = (\tau_{n+1}-\tau_n, u(\tau_n + \bigcdot))\;,
\end{equ}
where this time $u(\bigcdot) = (\theta(\bigcdot),z(\bigcdot))$ is the Markov process defined by
\eqref{e:equationTheta}.

Given any continuous function $F \colon \CX \to \R$, we can lift it to a function $\hat F \colon \hat \CX \to \R$ by
\begin{equ}[e:defFhat]
\hat F(\tau,u) = \int_0^\tau F(u(s))\,\dd s\;.
\end{equ}
Given an excursion $E$, we also write $\tau(E)$ for the time $\tau$ such that $E = (\tau,u)$. 
Note that the sequence $E_n$ given by \eqref{e:defEn} is Markovian as a consequence of the
Markov property of $u$.
Writing $\hat \mu_\alpha$ for its invariant measure 
on $\hat \CX$ (whose existence is shown in Lemma~\ref{lem:propertiesP} below and whose uniqueness
will not be used), 
it is then shown in Lemma~\ref{lem:idenAverages} that, for every 
$F\colon \CX \to \R$ with at most polynomial growth in the $z$ direction, one has
 $\hat F \in L^1(\hat \mu_\alpha)$ and
\begin{equ}
\int F(u)\,\mu_\alpha(\dd u) = \bar T_\alpha^{-1} \int_\CX \hat F(E)\,\hat \mu_\alpha(\dd E) \;,
\end{equ}
where $\bar T_\alpha = \int \tau(E)\,\hat \mu_\alpha(\dd E)$.
Write now $\hat P$ for the transition probabilities of the Markov chain $E_n$ on $\hat \CX$
and assume that we can find a function $G_\alpha \colon \CX \to \R_+$ such that, uniformly over $E$
and over $\alpha \ge 1$,
\begin{equ}[e:wanted]
\Big|\int \big(\hat F_\infty(E') - \hat F(E')\big)\,\hat P(E, \dd E')\Big| \le \int \hat G_\alpha(E')\,\hat P(E, \dd E')\;,
\end{equ}
with $F$ and $F_\infty$ as in \eqref{e:defFF}. This then immediately implies that 
\begin{equs}
\Big|\int \big(F_\infty - F\big)\,\dd \mu_\alpha\Big| &= \bar T_\alpha^{-1}\Big|\int_\CX \big(\hat F_\infty(E)-\hat F(E)\big)\,\hat \mu_\alpha(\dd E)\Big|  \\
&= \bar T_\alpha^{-1}\Big|\int_\CX\int_\CX \big(\hat F_\infty(E')-\hat F(E')\big)\,\hat P(E,\dd E')\,\hat \mu_\alpha(\dd E)\Big| \\
&\le \bar T_\alpha^{-1}\int_\CX\int_\CX \hat G_\alpha(E')\,\hat P(E,\dd E')\,\hat \mu_\alpha(\dd E)
= \int G_\alpha\,\dd \mu_\alpha\;.
\end{equs}
If we can choose $G_\alpha$ in such a way that furthermore 
\begin{equ}[e:wantedPropG]
\lim_{\alpha \to \infty} {1\over \sqrt\alpha}\int G_\alpha \,\dd \mu_\alpha = 0\;,
\end{equ}
then our claim follows.

We claim that for every $\kappa > 0$ one can find a constant $C > 0$ such that, setting
\begin{equs}
G_\alpha(z) &= C \alpha^\kappa\bigl(H_\alpha(z) + \alpha^{-2/3}|z|\bigr)\;, \\
H_\alpha(z) &= 
\left\{\begin{array}{cl}
	(1 \vee |z|) & \text{if $|z| \le \alpha^{2/3}$,} \\
	\alpha |z|^{-3/4} + \alpha^{1/3} + |z|^{1/3}(|z|/\alpha)^{16} & \text{if $|z| \ge \alpha^{4/5}$,} \\
	\alpha^{8/15} + \alpha^{2} |z|^{-2} & \text{otherwise,}
\end{array}\right.
\end{equs}
the bounds \eqref{e:wanted} and \eqref{e:wantedPropG} are satisfied.
Since a simple calculation shows that $\int G_\alpha \,d\mu_\alpha \lesssim \alpha^{\kappa + 1/3}$,
it only remains to show \eqref{e:wanted}.

Note that \eqref{e:wanted} holds provided we can show that, setting
$\tau = \inf\{t > 0\,:\, z(t) \not \in \CZ(z_0)\}$, one has the bound
\begin{equ}[e:realBound]
\E \Big|\int_0^\tau \big(F_\infty(u(t)) - F(u(t))\big)\,\dd t\Big| \le 
\underline G_\alpha(z_0) \,\E \tau \;,
\end{equ}
where we write $u = (\theta, z)$ and set
\begin{equ}[e:defGalphab]
\underline G_\alpha(z) \eqdef \inf_{z' \in \CZ(z)} G_\alpha(z')\;.
\end{equ}
In fact, since $G_\alpha$ has the property that 
$ G_\alpha(z) \lesssim \underline G_\alpha(z)$ uniformly over all $z$, \eqref{e:realBound}
is implied by the same bound with $\underline G_\alpha$ replaced by $G_\alpha$.
It remains to show 
that such a bound is indeed satisfied  for all initial conditions 
$(\theta_0,z_0)$. In order to show this, we consider a number of different regimes separately. 

\medskip\noindent\textbf{The case $\boldsymbol{|z_0| \le \alpha^{2/3}}$.}
This case is trivial since, provided that $C$ is large enough, one has $|F_\infty|+|F| \le G_\alpha$ in this region.

\medskip\noindent\textbf{The case $\boldsymbol{z_0 \le -\alpha^{4/5}}$.} 
We first note that by Fernique's theorem, combined with the upper bound in Lemma~\ref{lem:stopTime}
and Corollary~\ref{cor:negligible}, we can restrict ourselves to the event
\begin{equ}[e:goodSet]
\tau \le \alpha^{\delta} (z_0/\alpha)^2\;,\qquad \|W\|_{{1\over 2}-\delta} \le \alpha^\delta\;,
\end{equ}
where $\delta > 0$ is any (fixed, small) exponent and $\|\bigcdot\|_\alpha$ denotes the
$\alpha$-Hölder seminorm. This is because the event on which \eqref{e:goodSet} fails
has probability bounded by $C\alpha^{-p}$ for any $p$ (and a fortiori by $\alpha^{-2/3}$).

Let $N > 0$ be such that $\sqrt{|z_0|} 2^{-N} \in (1/2,1]$ and, for 
$n \in \{0,\ldots,N-1\}$, write $D_n \subset \R\P^1$ for the region defined by 
\begin{equ}
D_n = \{\theta \,:\, |\sin(\theta)| \in [2^{-n-1}, 2^{-n}]\}\;.
\end{equ}
We also set
\begin{equ}
D_N = \{\theta \,:\, |\sin(\theta)| \le 2^{-N}\}\;.
\end{equ}
Note that one has $|D_n| \approx 2^{-n}$ in the sense that the ratio between these quantities
is bounded from above and below by strictly positive constants. 
Note that for $n \in \{1,\ldots,N-1\}$, $D_n = D_n^+ \cup D_n^-$ consists of a pair of 
intervals placed symmetrically around
$\theta = 0$, while $D_0$ is a single interval centred around ${\pi\over 2}$ and $D_N$ is a 
single interval centred around $0$.

In this regime, the system spins around $\R\P^1$ many times: on each full spin, we exploit cancellations on every
pair $\{D_n^+, D_n^-\}$, via Lemma~\ref{lem:boundDiff}; however, depending on the starting point and the point at which the system
lies at the stopping time $\tau$, some intervals cannot be paired and contribute of a remainder bounded by at most two full spins.
This contribution will be taken care of at the end of the proof, in \eqref{e:lastCont}.

Using Lemma~\ref{lem:boundDiff}, we now have everything in place to bound the contribution 
in this regime
(here, `of order' means bounded above and below by a fixed multiple independent of $z_0$, $\alpha$, $n$).
Before time $\tau$, the right hand side 
\begin{equ}
F_n(\theta, t) = 1- z(t)\sin^2(\theta)
\end{equ}
for \eqref{e:equationTheta} in $D_n$ is of order $2^{-2n} |z_0|$ while
the integrand 
\begin{equ}
G_n(\theta,t) = -1 + {z(t)\over 2} \sin(2\theta)
\end{equ}
is at most of order $2^{-n}|z_0|$. (Recall that $F_\infty = 0$ in this regime and that 
$2^{-n}|z_0| \ge 1/4$ by our choice of $N$.) As a consequence, the time $t_n$ it takes to
cross one of the intervals is of order $t_n \sim 2^{-n} / F_n \sim 2^n/ |z_0|$.
(This also includes the intervals $D_0$ and $D_N$.) Summing up these bounds shows that the time it takes
to go around $\R\P^1$ once is of order $|z_0|^{-1/2}$.

We now apply Lemma~\ref{lem:boundDiff} for each of the pairs of intervals $\{D_n^+, D_n^-\}$,
except that that on $D_n^-$ we replace $\theta$ by $-\theta$ and run time backwards. 
We then have right hand sides $F_n$, $F_n'$ and integrands $G_n$, $G_n'$ with, as a consequence
of \eqref{e:goodSet},
\begin{equ}
d(F_n, F_n') \lesssim 2^{-2n} \alpha^{1+\delta} |z_0|^{-{1\over 2}({1\over2}-\delta)}\;,\quad
d(G_n, G_n') \lesssim 2^{-n} \alpha^{1+\delta} |z_0|^{-{1\over 2}({1\over2}-\delta)}\;.
\end{equ}
The total contribution over the two intervals is therefore bounded by 
\begin{equ}
2^{-n} \Bigl({G_n d(F_n, F_n') \over F_n^2} + {d(G_n, G_n') \over F_n}\Bigr)  
\lesssim 
\alpha^{1+\delta} |z_0|^{-1-{1\over 2}({1\over2}-\delta)}\;.
\end{equ}
Summing this and using the fact that $N \sim \log |z_0|$, we get a total contribution per round of 
\begin{equ}
\alpha^{1+\delta} |z_0|^{-1-{1\over 2}({1\over2}-\delta)}\log |z_0| 
\lesssim \alpha^{1+\delta} |z_0|^{\delta -{5\over 4}}\;.
\end{equ}
By \eqref{e:goodSet}, the total number of rounds in this regime is bounded by 
$\sqrt{|z_0|} (|z_0|/\alpha)^2 \alpha^\delta$, so  that the from the above bound we obtain a
total contribution of 
\begin{equ}\label{eq:totalcontr}
(|z_0|/\alpha)^2\alpha^{1+2\delta} |z_0|^{\delta -{3\over 4}}.
\end{equ}
We now distinguish two cases: 
if $|z_0|\leq \alpha$, then from Lemma~\ref{lem:stopTime} the total contribution is at most of order
\begin{equ}
(|z_0|/\alpha)^2  \alpha^{1+2\delta} |z_0|^{\delta -{3\over 4}}\lesssim \alpha^{2\delta+1}|z_0|^{\delta - {3\over 4}} \E \tau \lesssim
\alpha^{3\delta} \alpha |z_0|^{ - {3\over 4}} \E \tau  \;,
\end{equ}
which is consistent with \eqref{e:realBound} provided we choose $3\delta<\kappa$. Otherwise, if $|z_0|> \alpha$,  Lemma~\ref{lem:stopTime} implies that $1\lesssim \E\tau$. As a consequence, from \eqref{eq:totalcontr}
the total contribution is
\begin{align*}
(|z_0|/\alpha)^2\alpha^{1+2\delta} |z_0|^{\delta -{3\over 4}}
&\lesssim |z_0|^{1/3}(|z_0|/\alpha)^{\delta+11/12}\alpha^{\delta-1/12}\E \tau \\
&\lesssim |z_0|^{1/3}(|z_0|/\alpha)\E \tau \;,
\end{align*}
which implies \eqref{e:realBound} provided $\delta\leq 1/12$.
The additional contribution coming from the fact that some intervals may not be paired up is, for each of the $D_n$'s, of order at most
$G_n t_n \lesssim 1$, so the total contribution coming from this is at most
\begin{equ}[e:lastCont]
\log |z_0| \lesssim 
\begin{cases}
|z_0|^{\kappa-2} \alpha^2 \,\E \tau\lesssim |z_0|^{-{3\over 4}} \alpha^{1+\kappa} \,\E \tau, \quad &|z_0|\leq \alpha,\\
|z_0|^{1/3} \,\E \tau, \quad &|z_0|> \alpha,
\end{cases}
\end{equ}
which is again of the desired order.

\medskip\noindent\textbf{The case $\boldsymbol{-\alpha^{4/5} \le z_0 \le -\alpha^{2/3}}$.} 
In this case, we only have a contribution of $\log|z_0|$ for each round, and the expected number of rounds is bounded by 
$\sqrt{|z_0|} \;\E\tau+1$. Hence, the total contribution in this case is
\begin{align*}
\log|z_0| \left(\sqrt{|z_0|} \;\E\tau+1\right)\lesssim |z_0|^\kappa \left(\frac{\alpha}{z_0}\right)^2\E\tau\lesssim \alpha^\kappa \left(\frac{\alpha}{z_0}\right)^2\E\tau\;,
\end{align*}
as desired.

\medskip\noindent\textbf{The case $\boldsymbol{\alpha^{2/3} \le z_0 \le \alpha^{4/5+\kappa}}$.} 
This is similar to the previous case with the difference that $F_N$ is no longer bounded from below,
so it is more difficult to control the time spent in $D_N$. However, 
the integrand is bounded by $\sqrt{z_0}$ there, so the additional contribution 
coming from the time spent in  $D_N$ is at most $\sqrt{z_0}\, \E \tau \lesssim 
\alpha^{2/5+\kappa/2}\E \tau \le \alpha^{8/15+\kappa/2}\E\tau$ as required.
Finally, we have a contribution from $F_\infty$ in this regime, but since $F_\infty \lesssim \sqrt{z_0}$ 
we can treat this as an
error term which is of the same order as the previous contribution.

\medskip\noindent\textbf{The case $\boldsymbol{z_0 \ge \alpha^{24/23}}$.} 
This case is trivial since one has $|F_\infty| + |F| \le G_\alpha$.

\medskip\noindent\textbf{The case $\boldsymbol{\alpha^{4/5+\kappa} \le z_0 \le \alpha^{24/23}}$.} 
In this regime, we aim to show that most of the time the dynamic is very 
close to tracking the stable fixed point $\theta_+$. 
Write now $\tau_\star \le \tau$ for the minimum between $\tau$ and the first time when 
one has $|\theta(t) - \theta_+(z_0)| \le 1/(2 \sqrt{z_0})$.
It is then straightforward to show that $4\sqrt{z_0} \theta(t) \in [1,7]$ holds for all 
$t \in [\tau_\star, \tau]$. This is a consequence of the fact that 
$4\sqrt{z_0} \theta_+(z(t)) \in [2,6]$ (for example) for all $t \le \tau$ so that the interval
$[1,7]$ acts as a `trap' for $4\sqrt{z_0} \theta(t)$.

We then rewrite \eqref{e:equationTheta} as
\begin{equ}
\dot \theta = 1- z\sin^2(\theta) = 1- a^2(t) \theta^2\;,\qquad a^2(t) = z(t) {\sin^2(\theta) \over \theta^2}\;.
\end{equ}
Setting $f(t) = 1 + a(t)\theta(t)$, we are precisely in the situation of Corollary~\ref{cor:goodBound} below. By choosing $\alpha$ sufficiently large, we can 
guarantee that, for $t \in [\tau_\star, \tau]$, $|\theta|$ is sufficiently 
small so that ${\sin^2(\theta) \over \theta^2} > {1 \over 2}$, so that the assumptions
of the corollary are satisfied with $f_0 = 1$ and $a_0 = \sqrt{z_0}/2$.
Since $|\dot \theta| \lesssim 1$ and, provided that \eqref{e:goodSet} holds which we can assume
without loss of generality, $z$ satisfies
\begin{equ}
|z(t) - z(s)| \lesssim \alpha^{1+\delta} |t-s|^{{1\over 2}-\delta} + |t-s|z_0\;.
\end{equ}
Since $a_0 \approx \sqrt{z_0}$, a simple calculation shows that, on this event,
one can take the constant $K$ appearing
in Proposition~\ref{prop:boundStable} to be of order at most 
$\alpha^{1+\delta}z_0^{(2\delta-3)/4} \le \alpha^{1+2\delta}z_0^{-3/4}$. Since, by \eqref{e:defFP}, one furthermore has the bound 
\begin{equ}
|z\sin(2\theta_+(z)) - 2\sqrt z| \lesssim {1\over \sqrt z}
\end{equ}
in the region under consideration, we conclude that 
\begin{equ}
|z(t)\sin(2\theta(t)) - 2\sqrt {z(t)}| \lesssim {1\over \sqrt z_0} + \frac{\alpha^{1+2\delta}}{z_0^{3/4}} + \sqrt{z_0} \e^{-\sqrt{z_0} t/2}\;.
\end{equ}
This yields the bound 
\begin{equ}
\int_0^\tau |F_\infty(u(t)) - F(u(t))|\,\dd  t \lesssim 
\int_0^{\tau_\star} |F_\infty(u(t)) - F(u(t))|\,\dd t
+ 1 + \frac{\alpha^{1+2\delta}}{z_0^{3/4}}\tau\;.
\end{equ}
Choosing $\delta = \kappa/2$, we have
\begin{equ}
\E\Bigl(1 + \frac{\alpha^{1+2\delta}}{z_0^{3/4}}\tau\Bigr) \lesssim \Bigl(\frac{\alpha^{1+2\delta}}{z_0^{3/4}} + {\alpha^2 \over z_0^2}\Bigr)\E \tau
\lesssim {\alpha^{1+\kappa} \over z_0^{3/4}} \E \tau
\le G_\alpha(z_0) \E \tau\;,
\end{equ}
as required, so it remains to bound the first term.

For this, write 
\begin{equ}
\tau_1 = \inf\Big\{t>0\,:\, |\theta(t) - \theta_-(z(t))| \ge {1\over 2 \sqrt{z_0}}\Big\}\;.
\end{equ}
The contribution of $\int_{\tau_1}^{\tau_\star} |F_\infty(u(t)) - F(u(t))|\,dt$
is then bounded by $\CO(\log z_0)$ in exactly the same way as the contribution
 \eqref{e:lastCont} considered in the regime $z_0 \le -\alpha^{4/5}$.
 Writing similarly $\tau_0$ for the first time such that 
$|\theta(t) - \theta_-(z(t))| \ge z_0^{-2/3}$,
we can bound the contribution from
$\tau_0$ to $\tau_1$ by writing similarly to before
\begin{equ}
\dot \theta = 1- a^2(t) \theta^2 = f(t)\big(1 + a(t) \theta\big)\;,\qquad f(t) = 1-a(t)\theta(t)\;,
\end{equ}
and making use of the lower bound in Corollary~\ref{cor:goodBound}.
Since $K \lesssim \alpha^{1+2\delta}z_0^{-3/4}$ and $a_0 \approx \sqrt{z_0}$ as before, 
one has $z_0^{-2/3} \gg K/a_0^2$ so that the assumptions of the corollary are satisfied, which
shows that $|\tau_1 - \tau_0| \lesssim {\log z_0 \over \sqrt {z_0}}$.
The contribution of this regime is therefore bounded by $\log z_0$ as before.
 
To bound the contribution up to time $\tau_0$, we make use of Lemma~\ref{lem:expGrowth}.
Our aim is to use it in order to show that 
\begin{equ}[e:wantedBoundtau0]
\P \big(\tau_0 > \alpha^\kappa z_0^{-1/2}\big) \lesssim \alpha^{-2/3}\;,
\end{equ}
so that, as a consequence of Corollary~\ref{cor:negligible}, we obtain
\begin{equ}
\E\int_0^{\tau_0}|F_\infty(u(t)) - F(u(t))|\,dt 
\lesssim \alpha^\kappa + |z_0| \alpha^{\kappa-2/3} \E \tau
\le G_\alpha(z_0)\, \E \tau\;,
\end{equ}
as desired. As before, we can furthermore assume that we are on the event \eqref{e:goodSet}.
In particular, as long as $t < \alpha^\kappa z_0^{-1/2}$, we have the bound
\begin{equ}[e:boundzz0]
|z(t) - z_0| \lesssim \alpha^{1+\delta} t^{{1\over 2}-\delta} + |z_0| t
\lesssim \alpha z_0^{-1/5}\;,
\end{equ}
provided that both $\kappa$ and $\delta$ are sufficiently small.
In order to place ourselves in the framework of Lemma~\ref{lem:expGrowth}, we set
$\tilde \theta = \theta + 1/\sqrt z$, so that
\begin{equs}
\dd \tilde\theta &= 2\sqrt{z_0} \tilde \theta \,\dd t
+ \bigl(1 - z \sin^2(\tilde \theta - 1/\sqrt z) -2\sqrt{z} \tilde \theta \bigr)\,\dd t
+ 2(\sqrt z-\sqrt{z_0}) \tilde \theta \,\dd t \\
&\quad+ {\gamma \over 2\sqrt z} \,\dd t - {\gamma z_\star \over 2 z^{3/2}} \,\dd t + {3\alpha^2 \over 8 z^{5/2}}\,\dd t - {\alpha \over 2 z^{3/2}}\dd W\;.
\end{equs}
Note that there exists $\bar \eps$ sufficiently small such that if we 
consider times such that $|z-z_0| \le \bar \eps |z_0|$, we are in the setting of 
Lemma~\ref{lem:expGrowth} with $a = 2\sqrt {z_0}$ and $b = \alpha / (2z_0^{3/2})$,
so that $b \sqrt a \approx \alpha z_0^{-5/4}$.

Indeed, we note that since $\tilde \theta \lesssim z^{-2/3}$, we have the bound
\begin{equ}
\big|\sin^2(\tilde \theta - 1/\sqrt z) - (\tilde \theta - 1/\sqrt z)^2 \big| \lesssim 1/z^2\;,
\end{equ}
so that, at least for $t < \alpha^\kappa z_0^{-1/2}$, one has
\begin{equs}
\bigl|1 - z \sin^2(\tilde \theta - 1/\sqrt z) -2\sqrt{z} \tilde \theta \bigr|
&\lesssim |z\tilde \theta^2| +  {1\over z} \lesssim z^{-1/3}
\lesssim \alpha^{22/23} z_0^{-5/4}
\ll b \sqrt a\;,\\
|\sqrt z-\sqrt{z_0}| |\tilde \theta| &\lesssim \alpha z_0^{-41/30} \ll b \sqrt a\;,\\
{\gamma z_\star \over 2 z^{3/2}} \lesssim {\gamma \over 2\sqrt z} &\lesssim \alpha^{18/23}  z_0^{-5/4}\ll b \sqrt a\;,\\
{3\alpha^2 \over 8 z^{5/2}} &\lesssim  \alpha^{1-\kappa} z_0^{-5/4} \ll b \sqrt a\;.
\end{equs}
Applying Lemma~\ref{lem:expGrowth} with $K \approx z_0^{13/12} / \alpha$, we conclude 
that there exists a constant $C$ such that 
\begin{equ}
\P \Big(\tau_0 > {C \log z_0 \over \sqrt{z_0}}\Big) \le \alpha^{-2/3}\;,
\end{equ}
which indeed implies the required bound \eqref{e:wantedBoundtau0}.

\medskip\noindent\textbf{The case $\boldsymbol{\alpha}$ small and $\boldsymbol{z_\star < 1}$.} 
We now proceed to the proof of the second limit in in our statement. The methodology is the same as above,
but with $F_\infty = -1$, different choices of $G_\alpha$ and $\tau_n$, and the different terms are 
less delicate to estimate. 
This time, we set
\begin{equ}
\CZ(z) = 
\left\{\begin{array}{cl}
	[-2\alpha^{3/4},2\alpha^{3/4}] & \text{if $|z-z_\star| \le \alpha^{3/4}$,} \\
	\{z_\star + c(z-z_\star)\,:\, c \in [1/2,2]\} & \text{otherwise,}
\end{array}\right.
\end{equ}
and we write
\begin{equ}
\tau_{n+1} = (\tau_n + \alpha^{-2}) \wedge \inf\{t > \tau_n\,:\, z(t) \not \in \CZ(z(\tau_n))\}\;. 
\end{equ}
Finally, we set 
\begin{equ}[e:defGalpha]
G_\alpha(z) = 
\left\{\begin{array}{cl}
	C\alpha^{3/4} & \text{if $|z - z_\star| < \alpha^{3/4}/2$,} \\
	1+|z| & \text{otherwise.}
\end{array}\right.
\end{equ}
It is immediate that, if we define $\underline G_\alpha$ as in \eqref{e:defGalphab},
then $\int \underline G_\alpha\,\dd \mu_\alpha \lesssim \alpha^{3/4}$, so that the required bound
follows if we can show \eqref{e:realBound}.
Furthermore, \eqref{e:realBound} (with $F_\infty = -1$) clearly holds for $|z_0-z_\star| \ge \alpha^{3/4}$
since in that regime one has $|F(\theta,z)+1| \le |z| \le \underline G_\alpha(z_0)$ for all $z \in \CZ(z_0)$.

For $|z_0-z_\star| \le \alpha^{3/4}$, we note first that $\P(\tau < \alpha^{-2})$ decays like
$\exp(-c/\sqrt \alpha)$ as a consequence of standard Gaussian estimates, so we can assume that 
$\tau = \alpha^{-2}$. We also write $t_\star$ for the period of the ODE
\begin{equ}[e:simpleTheta]
\dot{\hat\theta} = 1-z_\star \sin^2(\hat \theta)\;,
\end{equ}
which is finite since $z_\star < 1$. By symmetry, we have
\begin{equ}
\int_0^{t_\star} \big(F(\hat \theta(t),z_\star)+1\big)\,\dd t = 0\;,
\end{equ}
independently of the initial condition $\hat \theta(0)$. We then break the interval $[0,\tau]$ into
chunks $[t_k,t_{k+1}]$ of length $t_\star$ and note that, by setting $\hat \theta_n$ to be the 
solution to~\eqref{e:simpleTheta} with $\hat \theta_n(t_n) = \theta(t_n)$, one has
\begin{equ}
\Big|\int_{t_k}^{t_{k+1}} \big(F(\theta(t),z(t))+1\big)\,\dd t\Big|
\le \int_{t_k}^{t_{k+1}} \big|F(\theta(t),z(t)) - F(\hat \theta(t),z_\star)\big|\,\dd t
\lesssim \alpha^{3/4}\;,
\end{equ}
except possibly for the last chunk which yields at most a contribution of order $1$.
It follows that we have 
\begin{equ}
\Big|\int_{0}^{\tau} \big(F(\theta(t),z(t))+1\big)\,\dd t\Big|
\lesssim 1 + \alpha^{{3\over 4}-2} \lesssim \alpha^{3/4} \E\tau \lesssim \underline G_\alpha(z_0)\E\tau\;,
\end{equ}
as required.

\medskip\noindent\textbf{The case $\boldsymbol{\alpha}$ small and $\boldsymbol{z_\star > 1}$.} 
We use the same definitions as in the previous case and note again that we only
need to consider the case $|z_0-z_\star| \le \alpha^{3/4}$ and $\tau = \alpha^{-2}$.
This time we note that as before the reference dynamic \eqref{e:simpleTheta} admits two fixed points
$\theta_\pm^\star$ such that $\sin \theta_\pm^\star = \pm 1/\sqrt{z_\star}$, so that
\begin{equ}
\sqrt{z_\star-1}-1 = F(\theta_+^\star,z_\star)\;.
\end{equ}
The claim therefore follows if we can show that, for $\tau = \alpha^{-2}$,
\begin{equ}[e:wantedBound]
\int_0^{\tau} \big|F(\theta(s),z(s)) - F(\theta_+^\star,z_\star)|\,ds \lesssim \alpha^{{3\over 4}-2}\;.
\end{equ}

\begin{remark}
To understand the behaviour of $\lambda_\alpha$ in this case, one could heuristically think as follows:
in the limit $\alpha\to 0$, one expects
the sequence of  invariant measures $\{\mu_\alpha\}_\alpha$ to accumulate more and more 
on the stable fixed point represented by the point-mass $\delta_{(\theta_+^\star,z_\star)}$.
From \eqref{e:defLalpha}, we see that $\lambda_0$ can then be computed explicitly,
leading to \eqref{e:lambalph}.
Estimate \eqref{e:wantedBound} makes  this heuristic precise by asserting that the system 
indeed spends most of the time near the stable fixed point.
\end{remark}

The proof of this bound follows the same lines as the proof of the case 
$\alpha^{4/5+\kappa} \le z_0 \le \alpha^{24/23}$ given above, except that the various
regimes are much easier to treat. Since $|\theta_+(z(t))-\theta_+^\star| \lesssim \alpha^{3/4}$
and since this is a stable fixed point (at least for $\alpha$ sufficiently small depending
on $z_\star$), it follows immediately that once $|\theta(t)- \theta_+^\star| \le C\alpha^{3/4}$
for a suitable $C$, this bounds holds for all subsequent times. 
Set $\tau_1 = \inf\{t > 0\,:\,  |\theta(t)- \theta_+^\star| \le C\alpha^{3/4}\}$
and break the integral in \eqref{e:wantedBound} into a contribution up to $\tau_1$ and a remainder.
Since the integrand in the remainder is bounded by $C\alpha^{3/4}$ and since we
are restricted to the event $\tau \le \alpha^{-2}$, this term is indeed bounded by $\CO(\alpha^{{3\over 4}-2})$
as required. 

Since on the other hand the integrand is always bounded by $\CO(1)$, it remains to show that 
\begin{equ}[e:wantedtau1]
\P(\tau_1 > \alpha^{3/4-2} ) \lesssim \alpha^{3/4}\;.
\end{equ}
Setting as before 
$\tau_0 = \inf\{t > 0\,:\,  |\theta(t)- \theta_-^\star| \ge C\alpha^{3/4}\}$, a very brutal bound
shows that $\tau_1 - \tau_0 \lesssim \alpha^{-3/4}$. This is simply because the right hand side
of the equation for $\theta$ is of size at least $\alpha^{3/4}$ during that time so that it takes time
at most $\CO(\alpha^{-3/4})$ to move by an order $1$ distance. It therefore remains to show
that a bound of the form \eqref{e:wantedtau1} holds for $\tau_0$, for which we would like
to apply Lemma~\ref{lem:expGrowth} again. 

If we set similarly to before $\tilde \theta(t) = \theta(t) - \theta_-(z(t))$,
then we see that it satisfies an equation of the form
\begin{equ}
\dd \tilde \theta = a \tilde \theta\,\dd t - {\gamma (z-z_\star)\over 2z_\star\sqrt{z_\star-1}}\,\dd t + G_1(\tilde\theta,z)\,\dd t + b\,G_2(z)\,\dd W\;,
\end{equ}
where
\begin{equ}
a = -2z_\star \sin\theta_-^\star\cos\theta_-^\star
= 2\sqrt{z_\star -1}\;,\qquad b = {\alpha \over 2z_\star\sqrt{z_\star-1}}\;,
\end{equ}
and where the nonlinearities $G_{1,2}$ satisfy the bounds
\begin{equ}[e:boundsG]
|G_1(\tilde \theta,z)| \lesssim \alpha^{3/2} + \alpha^{3/4}|\tilde \theta|\;,\qquad 
|G_2(z) - 1| \lesssim \alpha^{3/4}\;.
\end{equ}
Unfortunately, the second term appearing in the right hand side is of order $\alpha^{3/4}$,
which is much larger than $b\sqrt a = \alpha/z_\star$ so it appears that Lemma~\ref{lem:expGrowth}
doesn't apply. The trick is to look for a constant $c$ such that if we set
\begin{equ}
\hat \theta(t) = \tilde \theta(t) - c(z-z_\star)\;,
\end{equ}
this term cancels out. A simple calculation shows that this is the case if we choose
\begin{equ}
c = {\gamma \over 2z_\star \sqrt{z_\star-1} (\gamma+a)}\;.
\end{equ}
With this choice, $\hat \theta$ then satisfies 
\begin{equ}
\dd \hat \theta = \hat a \hat \theta\,\dd t  + \hat G_1(\hat\theta,z)\,\dd t + \hat b\,\hat G_2(z)\,\dd W\;,
\end{equ}
with $\hat a = a$, 
\begin{equ}
\hat b = {\alpha \over z_\star (2\sqrt{z_\star-1} + \gamma)}\;,
\end{equ}
and $\hat G_{1,2}$ satisfying the same bounds \eqref{e:boundsG} as $G_{1,2}$. We are now in the 
setting of Lemma~\ref{lem:expGrowth}. Since there exists a constant $\hat C$ such that 
$|\tilde \theta| > \hat C\alpha^{3/4}$ implies that $|\theta - \theta_-^\star| > C\alpha^{3/4}$,
the required bound on $\tau_0$ follows at once.

\medskip\noindent\textbf{The case $\boldsymbol{\alpha}$ small and $\boldsymbol{z_\star = 1}$.} 
This is the case where the two fixed points $\theta_\pm^\star$ are merged into one. We proceed again
in the same way as above with $G_\alpha$ as in \eqref{e:defGalpha}, but this time equal to $C\alpha^{1/4}$
for $|z-z_\star| \le \alpha^{3/4}/2$. Again, it suffices to consider the case $\tau = \alpha^{-2}$ and 
$|z_0-z_\star| \le \alpha^{3/4}$. Define the two intervals 
\begin{equ}
D_0 = \{\theta\,:\, |\theta - \theta_\pm^\star| \le C\alpha^{3/8}\}\;,\qquad
D_1 = \{\theta\,:\, |\theta - \theta_\pm^\star| \le \alpha^{1/8}\}\;,
\end{equ}
and write $D_1^c$ for the complement of $D_1$. The trajectory $t \mapsto \theta(t)$ can then
be partitioned into intervals according to whether $\theta(t) \in D_1$ (called `slow intervals') 
or $\theta(t) \in D_1^c$ (`fast intervals'). 
A simple comparison with scaled translates of the
ODE $\dot u = u^2$ shows that it takes at most a time of  order $\alpha^{-{1\over8}}$ to traverse 
a fast interval  and at least a time of order $\alpha^{- {3\over 8}}$ to go from
the boundary of $D_1$ to that of $D_0$ (and therefore \textit{a fortiori} also to traverse a slow interval).
Since one has $|F(\theta,z)| \lesssim 1$ in $D_1^c$ and 
$|F(\theta,z)| \lesssim |\theta - \theta_\pm^\star|^2 \lesssim \alpha^{1\over4}$ in $D_1$, we conclude that 
the average of $F$ over any time interval of size at least $\alpha^{-{3\over8}}$ is bounded
by $C \alpha^{1\over 4} + C \alpha^{{3\over8}-{1\over8}} \lesssim \alpha^{1\over4}$ as claimed.
\end{proof}

We complete this section with a bound on the stopping time $\tau$ given as in \eqref{e:deftaun} by 
\begin{equ}
\tau = \inf\{t > 0 \,:\, z(t) \not\in \CZ(z_0)\}\;,
\end{equ}
where $z$ denotes the solution to \eqref{e:defOU} with initial condition $z_0$.

\begin{lemma}\label{lem:stopTime}
For every $k \ge 1$ there exists a constant $c$ such that, for all $|z_0| \ge 1$ and all $\alpha \ge 1$,
one has $\E \tau \ge c \big(1 \wedge (z_0/\alpha)^2\big)$ and $({\E \tau^k})^{1/k} \le c^{-1} \big(1 \wedge (z_0/\alpha)^2\big)$.
\end{lemma}

\begin{proof}
We first show the lower bound. One has
\begin{equ}
{|z_0|\over 2} \le |z(\tau) - z_0| = \Big|\int_0^\tau z(s)\,\dd s + \alpha W(\tau)\Big|
\le 2|z_0| \tau + \alpha |W(\tau)|\;.
\end{equ}
It follows that either $\tau > {1\over 8}$ or $|W(\tau)| \ge |z_0|/(4\alpha)$, so that
$\tau \ge {1\over 8} \wedge (z_0/\alpha)^2 \sigma$, where the random variable $\sigma$ is equal in law to
\begin{equ}
\sigma = \inf\{s \ge 0\,:\, |W(s)| = 1/4\}\;.
\end{equ}
If $|z_0| \ge \alpha$, we then have
\begin{equ}
\E \tau \ge {1\over 8} \P\Big(\sup_{s \le 1/8} |W(s)| < 1/4\Big) \ge c\;,
\end{equ}
for some $c > 0$ as required. In the regime $\alpha \ge |z_0|$,
we have 
\begin{equ}
\E \tau \ge {1\over 8}(z_0/\alpha)^2 \E \bigl(\sigma \wedge (\alpha/z_0)^2\bigr)
\ge {1\over 8}(z_0/\alpha)^2 \E \bigl(\sigma \wedge 1\bigr) \ge c (z_0/\alpha)^2\;,
\end{equ}
for some (possibly different) strictly positive constant $c$ as required.

To show the upper bound, note first that we can assume without loss of generality that $z_0>0$.
For every stopping time $t \le \tau$ we then have, similarly to before,
\begin{equ}
z_0 \ge |z(t) - z_0| = \Big|\int_0^t z(s)\,\dd s + \alpha W(t)\Big|
\ge \alpha |W(t)| - 2|z_0| t \;,
\end{equ}
so that, setting $\eps = |z_0| / \alpha$, one has
\begin{equ}[e:upperBoundW]
|W(t)| \le \eps (1+2t)\;.
\end{equ}
It follows from the standard small ball estimates for Brownian motion \cite{Chung,ReviewLi} that, 
for all $t \le 1$, one has the bound
\begin{equ}
\P (\tau \ge t) \le C\exp \Bigl(- c{t  \over \eps^2}\Bigr)\;,
\end{equ}
for some constants $c,C> 0$.
On the other hand, for every $t \le \tau$, we also have
\begin{equ}
{z_0\over 2} \le z(t) \le z_0 \bigl(1- \f t2\bigr) + \alpha W(t)\;,
\end{equ}
which, by \cite[Eq.~3]{HallLinear}, implies that 
\begin{equ}
\P (\tau \ge t) \le \P\Big(W(t) \ge {\eps\over 2}(t-1)\Big) - \e^{\eps^2/2} \P\Big(W(t) \ge {\eps\over 2}(t+1)\Big)\;,
\end{equ}
so that in particular, for all $t \ge 2$,
\begin{equ}
\P (\tau \ge t) \lesssim (1 \wedge \eps^2) \exp(-\eps^2 t/8)\;.
\end{equ}
For $\eps \le 1/2$ it follows that, for a suitable $c > 0$,
\begin{equs}
\E \tau^k &= k \int_0^\infty t^{k-1} \P(\tau \ge t)\,\dd t \\
&\lesssim \int_0^1 t^{k-1} \e^{- c t / \eps^2}\,\dd t
+ \e^{- c / \eps^2}\int_1^{\eps^{-4}} t^{k-1} \,\dd t
+ \eps^2 \int_{\eps^{-4}}^\infty t^{k-1} \e^{-c\eps^2 t}\,\dd t\\
&\lesssim \eps^{2k} + \eps^{-4k} \e^{- c / \eps^2}
+ \eps^{2-2k} \e^{- c /\eps^2}\lesssim \eps^{2k}\;,
\end{equs}
as claimed. For $\eps \ge 1/2$, we have
\begin{equs}
\E \tau^k &\le 2 + k \int_2^\infty t^{k-1} \P(\tau \ge t)\,\dd t 
\lesssim 1 + \int_2^\infty t^{k-1} \e^{-c\eps^2 t}\,\dd t \\
&\lesssim 1 + \eps^{-k} \e^{-2c\eps^2} \lesssim 1\;,
\end{equs}
as claimed.
\end{proof}

\begin{corollary}\label{cor:negligible}
Let $A_\alpha$ be a collection of events such that $\P(A_\alpha) \le \alpha^{-2/3}$. Then, for every $\kappa > 0$ there exists a constant $C$ such that
\begin{equ}
\E \Bigl(\one_{A_\alpha} \int_0^\tau |F_\infty-F|\,\dd t\Bigr) \le C |z_0| \alpha^{\kappa -2/3}\,\E\tau\;.
\end{equ}
\end{corollary}

\begin{proof}
We have
\begin{equs}
\E \Bigl(\one_{A_\alpha} \int_0^\tau |F_\infty-F|\,\dd t\Bigr)
&\lesssim |z_0| \E \bigl(\one_{A_\alpha} \tau\bigr)
\le |z_0| \P(A_\alpha)^{(k-1)/k} (\E \tau^k)^{1/k} \\
&\lesssim |z_0| \P(A_\alpha)^{(k-1)/k} \E \tau\;,
\end{equs}
where we used Lemma~\ref{lem:stopTime} to get the last bound. The claim now follows 
at once by choosing $k$ sufficiently large.
\end{proof}

\appendix
\section{Useful bounds}

The following can be verified without encountering any surprises.
\begin{lemma}\label{lem:boundDiff}
Let $\theta_1, \theta_2$ be solutions to 
\begin{equ}
\dot \theta_i = F_i(\theta_i, t)\;,\qquad \theta_i(0) = a\;.
\end{equ}
for Lipschitz functions $F_i$ such that $F_i(\theta,t) > 0$. Write $\tau_i = \inf\{t > 0\,:\, \theta_i(t) = b\}$.
Then, for any two bounded continuous functions $G_i$ one has the bound
\begin{equ}
\Big|\int_0^{\tau_1} G_1(\theta_1,t)\,\dd t - \int_0^{\tau_2} G_2(\theta_2,t)\,\dd t\Big|
\le |b-a| {d(G_1,G_2) \underline F_2 + \overline G_2 d(F_1,F_2) \over \underline F_1\underline F_2}\;,
\end{equ}
where $\underline F$ (resp.~$\overline F$) denote the minimum (resp.~maximum) of $F$ and
$d(H_1,H_2) = \sup_{t_1 \le \tau_1,t_2 \le \tau_2,\theta} |H_1(\theta,t_1) - H_2(\theta,t_2)|$.\qed
\end{lemma}

\begin{lemma}\label{lem:expGrowth}
There exists a universal constant $\eps > 0$ such that the following holds.
Let $W$ be a standard Wiener process, let $E$ and $C$ be continuous processes adapted to the filtration generated by $W$, and let $x$ solve
\begin{equ}
\dd x = ax\,\dd t + E(t)\,\dd t + b\,C(t)\,\dd W(t)\;,
\end{equ}
for some constants $a,b > 0$.
Assume that, almost surely, one has $|C(t)-1| \le \eps$ and $|E(t)| \le \eps b\sqrt a$ for all $t>0$ and write $\tau_K$ for
the first time when $|x| \ge K b / \sqrt a$.
Then, the bound
\begin{equ}
\P \Big(\tau_K > {N \over a} \log{K\vee 1 \over \eps} \Big) \lesssim 2^{-N}
\end{equ}
holds uniformly over $x_0 \in \R$, $K > 0$ and every integer $N > 0$.
\end{lemma}

\begin{proof}
By considering $\hat x(t) = {\sqrt a \over b} x(t/a)$ we can reduce ourselves to the case $a = b = 1$, which we assume from now on. We write
\begin{equ}
x(t) = x_0 \e^t + Z(t) + \hat E(t)\;, 
\end{equ}
where
\begin{equ}
Z(t) = \int_0^t \e^{t-s} \dd W(s)\;,\qquad
\hat E(t) = \int_0^t \e^{t-s} \bigl(E(s)\,\dd s + (C(s)-1)\,\dd W(s)\bigr)\;.
\end{equ}
It follows that for any deterministic time $t > (\log 2)/2$ we can write
\begin{equ}[e:formSol]
x(t) = \e^t(x_0 + \eta + \eps \zeta)\;,
\end{equ}
where $\eta$ is a Gaussian random variable with variance in $[{1\over 4}, {1\over 2}]$
and, by Bernstein's inequality \cite[Ex.~IV.3.16]{RevuzYor}, 
$\zeta$ is a random variable (correlated with $\eta$ in general) with the property
that $\P(|\zeta| > 1+M) \le 2\exp(-M^2)$.

Set now $t = \log (K\vee 1)- \log \eps$ so that $\e^t \ge K/\eps$. Provided that $\eps < 1/\sqrt 2$, 
this yields the bound
\begin{equs}
\P \Big(\tau_K > \log {K\vee 1\over \eps} \Big) &\le \P (|x(t)| < K) \le \P (|x_0 + \eta + \eps \zeta| < \eps) \\
&\le \P (|x_0 + \eta| < M \eps) + \P (|\zeta| > M-1) \\
&\le {4M \eps \over \sqrt{2\pi}} + \P (|\zeta| > M-1)\;.
\end{equs}
It immediately follows that, by first choosing $M$ large enough so that 
$\P (|\zeta| > M-1) < 1/4$ and then choosing $\eps$ small enough so that 
${4M \eps \over \sqrt{2\pi}} < 1/4$, we can guarantee that $\P \big(\tau_K > \log \big((1\vee K) / \eps\big)\big) < 1/2$.

Using the same argument, combined with the fact that $E$ and $C$ are adapted, we obtain
the almost sure bound
\begin{equ}
\P \big(\tau_K > t + \log \big((1\vee K) / \eps\big)\,\big|\,\CF_t\big) < 1/2\;.
\end{equ}
Iterating this bound then yields the claim.
\end{proof}

\begin{proposition}\label{prop:boundStable}
Let $x$ and $y$ be solutions to 
\begin{equ}
\dot x = 1-a(t)\,x\;,\qquad 
\dot y = 1 +a(t)\,y\;,
\end{equ}
for a function $a$ such that $a(t) \ge a_0 > 0$ for all $t > 0$ and such that 
$|a(t) - a(s)| \le K$ whenever $|t-s| \le 1/a_0$. There exists a universal constant $C$ such that,
setting $x^\star(t) = 1/a(t)$, one has
\begin{equ}[e:goodBound]
|x(t) - x^\star(t)| \le |x(0) - x^\star(0)|\e^{-a_0 t} +  {CK \over a_0^{2}}\;,
\end{equ}
for all $t \ge 0$. Setting $y^\star(t) = -1/a(t)$, there exists a constant $C$ such that,
provided that $|y(0) - y^\star(0)| \ge C K/a_0^2$
and that $a(t) \in [a_0,2a_0]$ for all $t > 0$, one has
\begin{equ}[e:lowerBound]
|y(t) - y^\star(t)| \ge |y(0) - y^\star(0)|\e^{a_0 t/2}/2\;.
\end{equ}
\end{proposition}

\begin{proof}
Writing $\tilde x(t) = x(t) - x^\star(0)$, we have
\begin{equ}
\dot{\tilde x} = - a(t) \tilde x + (a(0) - a(t))\,x^\star(0)\;,
\end{equ}
so that, for $t \le 1/a_0$,
\begin{equ}
|\tilde x(t)| \le \e^{-a_0 t}|\tilde x(0)| + K |x^\star(0)| / a_0\;.
\end{equ}
On the other hand, one has
\begin{equ}
|x^\star(t) - x^\star(0)| = {|a(t) - a(0)|\over a(0)a(t)} \le {K  \over a_0^2}\;.
\end{equ}
Since $|x^\star(0)|\le 1/a_0$, we conclude that \eqref{e:goodBound} holds for 
$t \le a_0$ with $C = 2$. It then suffices to iterate this bound. 

Regarding the lower bound on $y$, a similar calculation shows that, for $t \le 1/a_0$,
\begin{equ}
|\tilde y(t)| \ge \e^{a_0 t} |\tilde y(0)| - K |y^\star(0)|\int_0^t \e^{2a_0 s}\,\dd s
\ge \e^{a_0 t} |\tilde y(0)| - (\e^2-1){K \over a_0^2} \;.
\end{equ}
Choosing $C$ sufficiently large and iterating this bound proves the claim.
\end{proof}

\begin{corollary}\label{cor:goodBound}
Let $x$ and $y$ be the solutions to 
\begin{equ}
\dot x = f(t)\big(1-a(t)\,x\big)\;,\qquad
\dot y = f(t)\big(1+a(t)\,y\big)\;,
\end{equ}
with $a$ as in Proposition~\ref{prop:boundStable} and $f \ge f_0$. 
Then, the conclusions still holds, but with $\e^{-a_0 t}$ and $\e^{a_0 t/2}$ 
replaced by $\e^{-a_0 f_0 t}$ and $\e^{a_0 f_0 t/2}$ respectively.
\end{corollary}
\begin{proof}
It suffices to perform a time change to eliminate $f$.
\end{proof}

\section{Path decomposition}

In this appendix, we justify the relation between the Markov process determined by the diffusion \eqref{e:equationTheta} and the Markov chain $E_n$ on path segments defined in \eqref{e:defEn}.
The natural topology on $\hat \CX$ (which turns it into a Polish space) is the one 
given by the metric 
\begin{equ}
d(\tau,u; \bar \tau, \bar u) = |\tau - \bar \tau| + \sup_{t \ge 0} |u(t\wedge \tau) - \bar u(t\wedge \bar \tau)|\;.
\end{equ}
In view of \eqref{e:defZz}, we can define for $n \ge -1$ the closed subset 
$\hat \CX_n \subset \hat \CX$ given for $n \ge 0$ by those trajectories $(\tau,u)$
such that $|z(0)| = 2^n$ and $z(\tau) \in \{2z(0), z(0)/2\}$. Similarly, 
$\hat \CX_{-1}$ is defined by imposing $|z(0)| = 1/2$ and $|z(\tau)| = 1$.

Writing $\hat \CP$ for the transition probabilities of $E_n$, we then have 
$\hat \CP(E, \hat \CX_{n-1} \cup \hat \CX_{n+1}) = 1$ for all
$E \in \hat \CX_n$ for $n \ge 0$ and $\hat \CP(E, \hat \CX_{0}) = 1$ when
$E \in \hat \CX_{-1}$. As a consequence, we can restrict our state space
and assume henceforth that $\hat \CX = \bigcup_{n \ge -1} \hat \CX_n$. We
also introduce the observable $N \colon \hat \CX \to \R$ such that
$N(E) = n$ for all $E \in \hat \CX_n$. One then has the following straightforward 
result.

\begin{lemma}\label{lem:propertiesP}
The Markov transition kernel $\hat \CP$ has the following properties.
\begin{enumerate}
\item It has the Feller property on $\hat \CX$.
\item For every $\alpha > 0$, the function $V(E) = \exp(\alpha N(E))$ 
is a Lyapunov function
for $\hat \CP$ in the sense that there exists $c < 1$ and $C > 0$ such that  
\begin{equ}
\int V(E') \hat \CP(E, dE') \le c V(E) + C\;.
\end{equ}
\item For every $\eps \ge 0$ and every $n \ge -1$ there exists a compact set 
$K \subset \hat \CX$ such that $\hat \CP(E,K) \ge 1-\eps$, uniformly over all
$E \in \hat \CX_n$. 
\end{enumerate}
In particular, it admits an invariant measure $\hat \mu_\alpha$ 
such that $V \in L^1(\hat \mu_\alpha)$.
\end{lemma}

\begin{proof}
The first property follows immediately from the fact that solutions to 
\eqref{e:equationTheta} depend continuously on their initial condition, combined with 
\cite[Cor.~8.4.2]{BogachevMT} and the fact
that their law does not charge the points of discontinuity of the map $u \mapsto \inf\{t>0\,:\, |z(t)| = a\}$ for any $a \neq z(0)$.

For the second property, a simple comparison with Brownian motion with drift shows that, 
for any fixed $\eps > 0$ and any $n$ large enough, one has 
$\hat \CP(E, \{N=n-1\}) \ge 1-\eps$ for all $E \in \hat \CX_n$, whence the claim 
follows.

The last property follows from the upper bound of Lemma~\ref{lem:stopTime} 
and Arzel\`a--Ascoli, 
using the fact that the $1/3$-H\"older norm (say) of a Wiener process admits
exponential moments over any fixed time interval. 

Combining the second and third part, we conclude that 
the sequence $\{\hat \CP^k(E,\cdot\,)\}_{k \ge 1}$ is tight for every $E \in \hat \CX$,
so that the existence of an invariant measure $\hat \mu_\alpha$ 
for $\hat \CP$ follows from the Krylov--Bogolyubov theorem. The integrability of $V$ is
an immediate consequence of part~2.
\end{proof}

Note that the invariant measure exhibited in this lemma is actually unique, 
but we will not use this and therefore do not give a proof. We then have 
the following result.

\begin{lemma}\label{lem:idenAverages}
Let $F \colon \CX \to \R$ be such that $|F(\theta,z)| \lesssim (1+|z|)^p$ for some
$p > 0$ and let $\hat F\colon \hat \CX \to \R$ be given by \eqref{e:defFhat}.
Then, one has
$\hat F \in L^1(\hat \mu_\alpha)$ and
\begin{equ}
\int F(u)\,\mu_\alpha(\dd u) = \bar T_\alpha^{-1} \int_\CX \hat F(E)\,\hat \mu_\alpha(\dd E) \;,
\end{equ}
where $\bar T_\alpha = \int \tau(E)\,\hat \mu_\alpha(\dd E)$.
\end{lemma}

\begin{proof}
The fact that $\hat F \in L^1(\hat \mu_\alpha)$ follows by combining
point 2.\ of Lemma~\ref{lem:propertiesP} with Lemma~\ref{lem:stopTime}.
The latter also shows that $\bar T_\alpha < \infty$.

Note now that it follows from Birkhoff's ergodic theorem that 
\begin{equ}[e:idenFhat]
\int_\CX \hat F(E)\,\hat \mu_\alpha(\dd E) = \lim_{n \to \infty} {1\over n}\sum_{k=1}^n \hat F(E_k)\;,
\end{equ}
where $E_k$ is the stationary Markov chain with transition probability $\hat \CP$ and
fixed time marginal $\hat \mu_\alpha$. However, the natural coupling \eqref{e:defEn} between
the chain $E_k$ and the solution to \eqref{e:equationTheta} yields
\begin{equ}
\sum_{k=1}^n \hat F(E_k) = \int_0^{\tau_n} F(u(s))\,\dd s\;,
\end{equ}
where $u(0)$ is distributed according to $\pi_0^* \hat \mu_\alpha$,
with $\pi_0(\tau,u) = u(0)$.
On the other hand, Propositions~\ref{prop:Hormander} and~\ref{prop:control}
imply that the time-$1$ map of \eqref{e:equationTheta} yields a Harris chain,
so that, arguing as in \cite[Prop.~17.1.6]{MeynTweedie}, the identity
\begin{equ}[e:Birkhoff]
\lim_{t \to \infty}{1\over t} \int_0^{t} F(u(s))\,\dd s = \int F(u)\,\mu_\alpha(\dd u)\;.
\end{equ}
holds almost surely. (The reason why we need to invoke the Harris chain property is that 
the initial condition $u(0)$ in \eqref{e:Birkhoff} is distributed according
to $\pi_0^* \hat \mu_\alpha$, which is singular with respect to $\mu_\alpha$.)

In particular, since $\lim_{n \to \infty} \tau_n = \infty$ almost surely
(take $F(\tau,u) = \tau$ in \eqref{e:idenFhat}), one has 
\begin{equs}
\int F(u)\,\mu_\alpha(\dd u) &= \lim_{n \to \infty}{1\over \tau_n} \int_0^{\tau_n} F(u(s))\,\dd s = \lim_{n \to \infty} {{1\over n} \int_0^{\tau_n} F(u(s))\,\dd s
\over {\tau_n \over n}} \\
&= \lim_{n \to \infty}{{1\over n} \sum_{k=1}^n \hat F(E_k) \over {1\over n} \sum_{k=1}^n  \tau(E_k)} =  \bar T_\alpha^{-1} \int_\CX \hat F(E)\,\hat \mu_\alpha(\dd E)\;,
\end{equs}
as claimed, where we used \eqref{e:idenFhat} in the last step.
\end{proof}

\endappendix

\bibliographystyle{./Martin}
\bibliography{./refs}

\end{document}